\documentclass[12pt]{amsart}
\newtheorem{theorem}{Theorem}[section]
\newtheorem{proposition}[theorem]{Proposition}
\newtheorem{lemma}[theorem]{Lemma}

\newtheorem{corollary}[theorem]{Corollary}
\newtheorem{remark}[theorem]{Remark}

\newcommand{\N}{\mathbb{N}}

\usepackage{geometry}                
\usepackage{graphicx}
\usepackage{amssymb}
\usepackage{epstopdf}
\DeclareMathOperator{\esssup}{esssup}
\DeclareMathOperator{\esinf}{essinf}

\DeclareMathOperator{\Cov}{Cov}

\numberwithin{equation}{section}

\title[]{Almost sure invariance principle for random dynamical systems via  Gou\"ezel's approach} 
\date{\today}

\begin{document}
 \maketitle
 \begin{center}
\authors{D.\ Dragi\v cevi\' c \footnote{Department of Mathematics, University of Rijeka, Rijeka Croatia. E-mail: {\tt \email{ddragicevic@math.uniri.hr}}.},
Y.\ Hafouta\footnote{Department of Mathematics, The Ohio State University, Columbus OH USA. E-mail: {\tt \email{yeor.hafouta@mail.huji.ac.il}, \email{hafuta.1@osu.edu}}}}
\end{center}


\begin{abstract}
We extend the spectral approach of S. Gou\"ezel  for the vector-valued almost sure invariance principle (ASIP)
to certain classes of non-stationary sequences with a weaker control over the behavior of the covariance matrices, assuming only linear growth.
Then we apply this extension
to obtain a quenched  vector-valued ASIP for  random perturbations of a fixed Anosov diffeomorphism as well as random perturbations of a billiard map associated to the periodic Lorentz gas. 
We also consider  certain classes of random piecewise expanding maps.
\end{abstract}

\section{Introduction}

\subsection{Almost sure invariance principle (ASIP)}
The almost sure invariance principle (ASIP) represents  a powerful statistical tool.
Given a sequence of vector-valued random variables $A_0,A_1,A_1,\ldots$,  it provides a coupling with an independent sequence of Gaussian random vectors $Z_0,Z_1,Z_2,\ldots$ such that
$$
\left|\sum_{j=0}^{n-1}A_j-\sum_{j=0}^{n-1}Z_j\right|=o(s_n),
$$ 
where $s_n=\|S_n\|_2$, $S_n=\sum_{j=0}^{n-1}A_j$, and the $L^2$-norm of the sum $\sum_{j=0}^{n-1}Z_j$ has the form $s_n(1+o(1))$.  We stress that ASIP implies several other important limit theorems, such as the central limit theorem and the functional central limit theorem. We refer to~\cite{PS} for a detailed discussion. 

\subsection{Spectral approach for ASIP}\label{SASIP}
As will be discussed below, many  proofs of the ASIP in the scalar case rely on an appropriate martingale approximation. In~\cite{GO}, S. Gou\"ezel developed a spectral method for proving the ASIP for classes of non-stationary random vectors which are bounded in $L^p$ for some $p>4$, satisfying certain mixing assumptions of a ``spectral type" and having a certain control\footnote{Roughly speaking, it is assumed that $\text{Cov}(S_n-S_m)\asymp (n-m)\Sigma^2$ for some positive definite matrix $\Sigma^2$.}  over the behavior of the covariance matrices of $S_n-S_m$ for $n>m$, where $S_n:= \sum_{k=0}^{n-1} A_k$. More recently, in \cite{DH1} we have extended this approach to real-valued sequences $\{A_j\}$ with the property that the variance of $S_n-S_m$ grows linearly fast in $n-m$.  
We stress that the purpose for this extension was to handle certain  induced random dynamical systems, and so apart from weakening the type of a  covariance control, we have established the version of ASIP when $\|A_n\|_{L^p}=O(n^{1/p})$. 

In this paper we present a second extension of  Gou\"ezel's approach (see Theorem~\ref{Gthm}), proving the ASIP under the spectral mixing assumptions for sequences bounded in $L^p$ with the property that the covariance matrix of $S_n$ grows linearly fast in $n$. We emphasize that we do not impose 
 any assumptions related to the behaviour of the covariance matrices of $S_n-S_m$. The purpose of this extension is to obtain the ASIP for several classes of random expanding or random hyperbolic dynamical systems (see Theorem~\ref{Th} and examples presented in Section~\ref{EX}). We note that in~ \cite{DH} we have proved several other limit theorems (such as the central limit theorem, Berry-Esseen bounds, Edgeworth expansions, the local limit theorem and several versions of a large deviations principle)  for vector-valued observables and  essentially the same classes of random dynamical systems as studied in the present paper. However, ASIP remained out of reach of the techniques developed in~\cite{DH}. Conversely, the results in the present paper imply only the versions of the central limit theorem discussed in~\cite{DH}, while all other results in~\cite{DH} are completely independent of those developed in the present paper. 

\subsection{ASIP for deterministic dynamical systems}
We emphasize that the ASIP has been widely studied for deterministic dynamical systems. We in particular mention the works of
 Field, Melbourne and T\"or\"ok~\cite{FieldMelbourneTorok} as well as Melbourne and Nicol~\cite{MN1, MN2} (completed by Korepanov~\cite{KorepanovEq}), in which the authors obtained ASIP for wide classes of (nonuniformly) hyperbolic maps. In contrast to their approaches which relied on martingale techniques,  Gou\"ezel's~\cite{GO} spectral approach works for vector-valued sequences, and was applied in \cite{GO}  to certain classes of deterministic dynamical systems, with the property that the corresponding transfer operator has a spectral gap on an appropriate Banach space. In situations when  this method is applicable, it was pointed out in~\cite{GO} that it gives better error
rates in ASIP from those obtained in~\cite{MN1, MN2}. Finally, we mention the recent important papers by Cuny and Merlevede~\cite{CM},  Korepanov, Kosloff and Melbourne~\cite{KZM}, Korepanov~\cite{Korepanov}, as well as Cuny, Dedecker,  Korepanov and  Merlevede~\cite{CDKM, CDKM2} in which the authors
further improved the error rates in ASIP for a wide class of (nonuniformly) hyperbolic deterministic dynamical systems. 
\subsection{ASIP for random dynamical systems}
A random dynamical system is generated by random compositions 
\[T_\omega^{n}:=T_{\sigma^{n-1}\omega}\circ \cdots \circ T_{\sigma\omega}\circ T_\omega \quad \omega\in\Omega,  \ n\in \N,
\] of maps $T_\omega$ (acting on some space $M$) which are 
driven by an invertible, measure-preserving transformation $\sigma$ on some probability space $(\Omega, \mathcal F, \mathbb P)$. We stress that random dynamical systems are key tools to model many natural phenomena, including the transport
in complex environments such as in the ocean or the atmosphere~\cite{Arnold}.

To the best of our knowledge, the ASIP in the context of random dynamical systems was first discussed by Kifer~\cite{kifer}. Indeed, in~\cite{kifer} it was mentioned that the techniques developed there can be used to obtain  the scalar-valued quenched ASIP for random expanding dynamics. More recently, the \emph{annealed} ASIP was obtained for several classes of random dynamical systems~\cite{ANV,STE,STSU}. In these works, the idea is to consider the transfer operator associated to the skew-product transformation (see~\eqref{spt}) and to show that it has a spectral gap on an appropriate space. Then, it remains to apply Gou\"ezel's ASIP~\cite{GO}. However, in order for this method to work, it is necessary to make strong (mixing) assumptions on the base space $(\Omega, \mathcal F, \mathbb P)$ of the random dynamical system (in all mentioned papers, $(\Omega, \mathcal F, \mathbb P)$ is a Bernoulli shift). On the other hand, in \cite{HNTV} the authors proved a scalar-valued ASIP for certain classes of sequential  expanding or hyperbolic dynamical systems with some assumptions on the growth rates of the variances. The approach in \cite{HNTV} relied on an appropriate  martingale approximation combined with the results from~\cite{CM}.
Moreover, in~\cite{DFGTV1} (by building on the approach developed in~\cite{CM, HNTV}) the authors have obtained the quenched scalar-valued ASIP for certain classes of random piecewise expanding dynamics without any mixing assumptions for the base space. 
Finally, we mention two recent papers \cite{Su1, Su2} by Y. Su  devoted to the ASIP for certain classes of random expanding maps and maps which admit a random tower extension. We note that while the latter two papers concern random dynamics with sub-exponential decay of correlations, the ASIP rates obtained there are not as good as the ones in \cite{DFGTV1} and the present paper.


\subsection{Contributions of the present paper}
As we have already mentioned in Subsection~\ref{SASIP},  in the present paper we establish (see Theorem~\ref{Gthm}) an appropriate modification of~\cite[Theorem 1.3.]{GO}, and use it to establish the quenched ASIP for several classes of random dynamical systems (see
Theorem~\ref{Th} and examples given in Section~\ref{EX}).
More precisely, we consider the following three cases:
\begin{itemize}
\item maps $T_\omega$, $\omega \in \Omega$ are Anosov diffeomorphisms on a compact Riemannian manifold $M$ that belong to a sufficiently small neighborhood of a fixed Anosov diffeomorphism $T$ on $M$;
\item maps $T_\omega$, $\omega \in \Omega$ are suitable  perturbations of a billiard map associated to the periodic Lorentz gas studied by Demers and Zhang~\cite{DZ};
\item $(T_\omega)_{\omega \in \Omega}$ is a family of piecewise expanding maps on the unit interval satisfying appropriate conditions as in~\cite{DFGTV1, DFGTV2}.
\end{itemize}
For a  sufficiently regular  random vector-valued observable $g_\omega:X\to \mathbb{R}^d$, $\omega\in\Omega$, our quenched ASIP  implies that  for $\mathbb P$-a.e. $\omega \in  \Omega$,  the random Birkhoff sums $\sum_{j=0}^{n-1}g_{\sigma^j\omega}\circ T_\omega^{j}$ can be approximated in the strong sense by a sum of Gaussian independent random vectors $\sum_{j=0}^{n-1}Z_j$ (depending on $\omega$) , with the error being negligible compared to $n^{\frac12}$. In fact, we will show that for every $\epsilon>0$, the error  term is at most $o(n^{1/4+\epsilon})$. In particular,  we extend the scalar ASIP for random expanding maps from \cite{HNTV, DFGTV1} to vector-valued observables. Furthermore, we for the first time obtain the quenched ASIP for some classes of random hyperbolic dynamics. 

Finally, let us explain why our modification of~\cite[Theorem 1.3]{GO} is needed. We observe that~\cite[Theorem 1.3]{GO} can be applied 
when the asymptotic covariance matrix
$$
\Sigma^2=\lim_{n\to\infty}\frac 1n\text{Cov}(S_n)
$$
exists and 
\begin{equation}\label{SC}
\Big|\text{Cov}(S_n-S_m)-(n-m)\Sigma^2\Big|\leq C(m-n)^{\alpha},
\end{equation}
for some $0<\alpha<1$ small enough. In~\cite{GO}, it is explained that this condition is satisfied for wide classes of \emph{deterministic} dynamical systems with some expansion or hyperbolicity. However, 
 from a general non-stationary point of view the above condition is strong. From a probabilistic point of view, random dynamical systems are examples of non-stationary processes with an asymptotic covariance matrix
\[
\Sigma^2:=\lim_{n\to\infty}\frac 1n \Cov_\omega\left(\sum_{j=0}^{n-1}g_{\sigma^j\omega}\circ T_\omega^{j}\right).
\]
However, the rate of the convergence in the above expression depends on $\omega$, since the covariance matrix of  random Birkhoff sums is controlled by certain ergodic averages.
Thus,  \eqref{SC} may fail to hold in a random setup. 

Our modification of~\cite[Theorem 1.3]{GO} is precisely tailored to overcome this difficulty. We stress that Theorem~\ref{Gthm} is completely different from~\cite[Theorem 1]{DH}, and that we expect that a combination of those two results  will lead to an extension of \cite[Theorem 1.3]{GO} not only in the absence of~\eqref{SC}, but also when the underlying sequence $\{A_n\}$ satisfies $\|A_n\|_p=O(n^{b_p})$ for some  constant $b_p>0$ and $p>4$.
\subsection{Organization of the paper}
The paper is organized as follows: in Section~\ref{Sec6} we formulate the main result of the present paper (see Theorem~\ref{Gthm}), which gives an abstract ASIP result for certain classes of non-stationary sequences.
The proof of Theorem~\ref{Gthm} is presented in Section~\ref{PROOF}.  In Section~\ref{Random}, we give  sufficient conditions under which Theorem~\ref{Gthm} can be applied in the context of random dynamics. In particular, we establish Theorem~\ref{Th} which gives an abstract version of the ASIP for random dynamical systems. 
Finally, in Section~\ref{EX} we discuss concrete examples of random dynamics to which Theorem~\ref{Th} is applicable.

\section{Gou\"ezel's ASIP in the nonstationary setup}\label{Sec6}
The purpose of this section is to provide a  certain modified version of~\cite[Theorem 1.3.]{GO}.

Let $(A_1, A_2, \ldots )$ be an $\mathbb R^d$-valued process on some probability space $(\Omega, \mathcal F, \mathbb P)$, where $d\in \N$. We will denote the scalar product of two vectors $t$ and $v$ in $\mathbb R^d$ by $t\cdot v$, and when it is more convenient we will also abbreviate and just write $tv$. We will also denote the Euclidean norm of a vector $v$ by $|v|$.
We first recall the  condition introduced in~\cite{GO}: there exists $\varepsilon_0>0$ and $C,c>0$ such that for any $n,m\in \N$, $b_1<b_2< \ldots <b_{n+m+1}$, $k\in \N$ and $t_1,\ldots ,t_{n+m}\in\mathbb R^d$ with $|t_j|\leq\varepsilon_0$, we have that
\begin{equation}\label{HH}
\begin{split}
&\Big|\mathbb E\big(e^{i\sum_{j=1}^nt_j(\sum_{\ell=b_j}^{b_{j+1}-1}A_\ell)+i\sum_{j=n+1}^{n+m}t_j(\sum_{\ell=b_j+k}^{b_{j+1}+k-1}A_\ell)}\big) \\
&-\mathbb E\big(e^{i\sum_{j=1}^nt_j(\sum_{\ell=b_j}^{b_{j+1}-1}A_\ell)}\big)\cdot\mathbb E\big(e^{i\sum_{j=n+1}^{n+m}t_j(\sum_{\ell=b_j+k}^{b_{j+1}+k-1}A_\ell)}\big)\Big| \\
&\le C(1+\max|b_{j+1}-b_j|)^{C(n+m)}e^{-ck}.
\end{split}
\end{equation}

Our extension of~\cite[Theorem 1.3]{GO} is the following result. 
\begin{theorem}\label{Gthm}
Let $\{A_n\}$ be a centered sequence of $\mathbb R^d$-valued random variables which is bounded in $L^p$ for some $p>4$, and satisfies property~\eqref{HH}. Assume, in addition, that there exists a constant $c_1>0$ so that for every sufficiently large $n$ and $v\in\mathbb R^d$ we have that 
\begin{equation}\label{CovCond}
\Cov \bigg{(}\sum_{j=1}^{n}A_j \bigg{)}v\cdot v\geq c_1n|v|^2.
\end{equation}
Then, there exists a  coupling between $\{A_n\}$ and a sequence of independent and centered Gaussian 
$d$-dimensional random vectors $Z_1,Z_2,\ldots$ such that for all $\delta>0$ we have 
\begin{equation}\label{asip rate}
\left|\sum_{j=1}^{n}A_j-\sum_{j=1}^{n}Z_j\right|=o(n^{a_p+\delta})\quad\text{$\mathbb P$-a.s.},
\end{equation}
where $a_p=\frac p{4(p-1)}=\frac14+\frac 1{4(p-1)}$. Moreover, there exists a constant $C=C_\delta>0$ so that for every $n\geq1$,
\begin{equation}\label{Var est}
\left\|\sum_{j=1}^n A_j-\sum_{j=1}^n Z_j\right\|_{L^2}\leq Cn^{a_p+\delta}.
\end{equation}
Finally, if there are $C_0>0$ and $r\in (0,1)$ so that for every unit vector $v\in\mathbb R^d$ and $n,k\in\mathbb N$ we have 
\begin{equation}\label{UniDec}
\left|\text{Cov}(A_n\cdot v, A_{n+k}\cdot v)\right|\leq C_0r^k,
\end{equation}
then there is a constant $C=C_\delta>0$ so that for all $n\geq1$  and a unit vector $v\in\mathbb R^d$,
\begin{equation}\label{Var est1}
\left\|\sum_{j=1}^n A_j\cdot v\right\|_{L^2}^2-Cn^{2a_p+\delta}\leq \left\|\sum_{j=1}^n Z_j\cdot v\right\|_{L^2}^2\leq \left\|\sum_{j=1}^n A_j\cdot v \right\|_{L^2}^2+Cn^{2a_p+\delta}.
\end{equation}
\end{theorem}
\begin{remark}\label{BW Rem}
In the scalar case $d=1$, \eqref{Var est} and the linear growth of the variance of $\sum_{j=1}^n A_j$ yields that the difference between the variances of $\sum_{j=1}^n Z_j$ and $\sum_{j=1}^n A_j$ is only $O(n^{a_p+\frac12+\delta})$, which is close to $O(n^{3/4+\delta})$ when $p$ is large. However,  \eqref{Var est1} yields that the difference between the variances is $O(n^{2a_p+\delta})$, which is close to $O(n^{\frac12+\delta})$ when $p$ is large. Moreover, using \eqref{Var est1} together with~\cite[Theorem 3.2 A]{HR}, we conclude that under \eqref{UniDec} in the scalar case, for any $\delta>0$ there is a coupling of $\{A_n\}$ with a standard Brownian motion $\{W(t):\,t\geq0\}$ such that
\begin{equation}\label{ZZZ}
\left|\sum_{j=1}^{n}A_n-W(V_n)\right|=o(n^{\frac14+\frac{1}{4(p-1)}+\delta})\quad\text{$\mathbb P$-a.s.},
\end{equation}
where $V_n=\text{Var}\left(\sum_{j=1}^n A_j\right)$.
Using only \eqref{Var est} and~\cite[Theorem 3.2 A]{HR}, we could have only concluded that the left hand side in~\eqref{ZZZ} is of magnitude $o(n^{a_p/2+\frac14+\delta})$, which is close to $n^{3/8}$ when $p$ is large, and not to $n^{1/4}$.
\end{remark}

\section{Proof of Theorem \ref{Gthm}}\label{PROOF}
The proof  is  a modification of the proof of Theorem 1.3 in \cite{GO}.

We consider the so-called  big and small blocks as  introduced in~\cite[p.1659]{GO}. Fix $\beta\in(0,1)$ and $\varepsilon\in(0,1-\beta)$. Furthermore, let  $f=f(n)=\lfloor \beta n \rfloor$. Then,  Gou\"ezel decomposes $[2^n,2^{n+1})$ into a union of $F=2^f$ intervals $(I_{n,j})_{0\leq j<F}$ of the same length, and $F$ gaps $(J_{n,j})_{0\leq j<F}$ between them. In other words, we have 
\[
[2^n, 2^{n+1})=J_{n,0}\cup I_{n, 0}\cup J_{n,1}\cup I_{n,1}\cup \ldots \cup J_{n, F-1}\cup I_{n, F-1}.
\]
Let us outline the construction of this decomposition. For $1\le j<F$, we write $j$ in the form $j=\sum_{k=0}^{f-1} \alpha_k(j)2^k$ with $\alpha_k \in \{0, 1\}$. We then take the smallest $r$ with the property that $\alpha_r(j)\neq 0$ and take $2^{\lfloor \epsilon n\rfloor}2^r$ to be the length of $J_{n,j}$. 
In addition, the length of $J_{n,0}$ is $2^{\lfloor \epsilon n\rfloor}2^f$. Finally, the length of each interval $I_{n,j}$ is $2^{n-f}-(f+2)2^{\lfloor \epsilon n\rfloor-1}$.

In addition,  we recall some notations from~\cite{GO} which we will also use.  We define a  partial order on  $\{(n,j):\,n\in\mathbb N,\,0\leq j<F(n)\}$ by writing  $(n,j)\prec (n',j')$ if the interval $I_{n,j}$ is to the left of $I_{n',j'}$. Observe that a sequence  $((n_k,j_k))_k$ tends to infinity for this order if and only if $n_k\to\infty$. Moreover, let 
\[
X_{n,j}:=\sum_{\ell\in I_{n,j}}A_\ell
\]
and \[\mathcal I:=\bigcup_{n,j}I_{n,j} \quad  \text{and} \quad \mathcal J:=\bigcup_{n,j}J_{n,j}.\]

Finally, let us recall~\cite[Proposition 5.1]{GO}.
\begin{proposition}\label{Prop5.1}
There exists a coupling between $(X_{n,j})$ and a family of independent random vectors $(Y_{n,j})$ such that  $Y_{n,j}$ and  $X_{n,j}$ are equally distributed and almost surely, when $(n,j)$ tends to infinity, 
\begin{equation}\label{Coup}
\left|\sum_{(n',j')\prec (n,j)}X_{n',j'}-Y_{n',j'}\right|=o(2^{(\beta+\varepsilon)n/2}).
\end{equation}
\end{proposition}

The first modification (in comparison to~\cite{GO}) that  we need is a certain $L^2$-version of Proposition~\ref{Prop5.1}.

\begin{proposition}\label{L2 Version}
There exists a coupling between $(X_{n,j})$ and $(Y_{n,j})$ from Proposition \ref{Prop5.1} such that (in addition to~\eqref{Coup}), for  some $C>0$ and all $n\in \N$ we have that
\[
\left\|\sum_{(n',j')\prec (n,j)}X_{n',j'}-Y_{n',j'}\right\|_{L^2}\leq C2^{\beta n/2}.
\]
\end{proposition}

\begin{proof}
We will show that we can couple $(X_{n,j})$ and $(Y_{n,j})$ so that 
\begin{equation}\label{L2est1}
\left\|\sum_{j=0}^{F(n)-1}(X_{n,j}-Y_{n,j})\right\|_{L^2}\leq C2^{\beta n/2}
\end{equation}
for every $n\in \N$, but it will be clear from the arguments in the proof that the same estimate holds true for the sum of $X_{n,j}-Y_{n,j}$, $j=0,1,...,j'$  where $j'<F(n)$. Using this extended version of \eqref{L2est1}, it is clear that Proposition \ref{L2 Version} follows.

Let $\tilde{X}_{n,j}=X_{n,j}+V_{n,j}$ and $\tilde{Y}_{n,j}=Y_{n,j}+V'_{n,j}$, where the $V_{n,j}$'s and  the $V_{n,s}'$ are independent copies of the symmetric random vector $V$ constructed in~\cite[Proposition 3.8]{GO}, which are independent of everything else and from each other (enlarging  our probability space if necessary). 
We will first return to the arguments in Step~1 of the  proof of~\cite[Theorem 1.3]{GO} (i.e. to the proof of Proposition \ref{Prop5.1}), and show that for every $s>4$ there exists a constant $C_s>0$ and a coupling between the $\tilde{X}_{n,j}$'s and the $\tilde{Y}_{n,j}$'s such that for all $0\leq j<F(n)$ we have that
\begin{equation}\label{APROX2}
\mathbb P(|\tilde{X}_{n,j}-\tilde{Y}_{n,j}|\geq C_s s^{-n})\leq C_s s^{-n}.
\end{equation}
Indeed, this was proved for $s=4$ in~\cite[Section 5]{GO} (see~\cite[p.1663]{GO}), but taking a careful look at the proofs of~\cite[Lemma 5.2]{GO} and~\cite[Lemma 5.4]{GO} one observes  that $4^n$ can be replaced with $s^n$, for any $s>4$, since the upper bounds there on the Prokhorov distances appearing in the proofs of these lemmas  have the form $e^{-\delta_1 2^{-n\delta_2}}$ for some positive $\delta_1$ and $\delta_2$. 

Next, set $\Gamma_j=\Gamma_{j,n,s}=\{|\tilde{X}_{n,j}-\tilde{Y}_{n,j}|\geq C_s s^{-n}\}$, where $0\leq j<F(n)$.
Then, by applying the Cauchy-Schwartz  inequality,  we have that 
\[
\begin{split}
\left\|\sum_{j=0}^{F(n)-1}(\tilde{X}_{n,j}-\tilde{Y}_{n,j})\right\|_{L^2}^2 &\le
2\mathbb P^2(\cap_{j}\Gamma_j^c)(C_sF(n)s^{-n})^2+2\mathbb E\big[\mathbb I_{\cup_j\Gamma_j}\big (\sum_{j=0}^{F(n)-1}(\tilde{X}_{n,j}-\tilde{Y}_{n,j})\big)^2\big] \\
&\le 2(C_s2^{\beta n}s^{-n})^2+2\mathbb P^{\frac 12}(\cup_j\Gamma_j)\left\|\sum_{j=0}^{F(n)-1}(\tilde{X}_{n,j}-\tilde{Y}_{n,j})\right\|_{L^4}^2 \\
&=:I_1+I_2,
\end{split}
\]
where $\mathbb I_\Gamma$ denotes the indicator function of a set $\Gamma$, $\Gamma^c$ denotes its complement in the underlying probability space and $\mathbb P^r(\Gamma)=\big( \mathbb P(\Gamma)\big)^r$ for  $r>0$.
Observe that when $s>4$,  we have that 
\[
I_1\leq C4^{-n},
\]
for some $C>0$. 
On the other hand, since the $L^4$-norms of $\tilde{X}_{n,j}$ and $\tilde{Y}_{n,j}$ are bounded by $c|I_{n,j}|$ (where $c$ is some constant), then
\[
\left\|\sum_{j=0}^{F(n)-1}(\tilde{X}_{n,j}-\tilde{Y}_{n,j})\right\|_{L^4}^2\leq 
\bigg(2c\sum_{j=0}^{F(n)-1}|I_{n,j}|\bigg)^2\leq C2^{2n}
\]
where $C>0$ is some constant. Note that  in the above inequality we have used that the sum of all the lengths of intervals $I_{n,j}$, $0\le j< F(n)$ does not exceed $2^{n+1}$.
The above inequality together with~(\ref{APROX2}) implies that 
\[
I_2\leq (F(n)C_s s^{-n})^{\frac12}\cdot C2^{2n}.
\]
Thus, $I_2$ is bounded in $n$ when $s>32$.

Finally, since $V$ is symmetric and the $V_{n,j}$'s are i.i.d., we have that 
\[
\left\|\sum_{j=0}^{F(n)-1}V_{n,j}\right\|_{L^2}^2=F(n)\|V\|_{L^2}^2\leq 2^{\beta n}\|V\|_{L^2}^2,
\]
which together with the above estimates on $I_1$ and $I_2$ completes the proof of the proposition.
\end{proof}

We now derive the following corollary. 

\begin{corollary}\label{CovarianceCor}
There exists  a constant $c>0$ such that for every sufficiently large $n\in \N$ and all $v\in\mathbb R^d$ we have that 
\begin{equation}\label{CovCond2}
\Cov\bigg(\sum_{m=1}^n\sum_{j=0}^{F(m)-1}Y_{m,j}\bigg)v\cdot v\geq c 2^n|v|^2.
\end{equation}
\end{corollary}

\begin{proof}
Firstly, observe that for any two centered random vectors $X=(X_i)_{i=1}^d$ and $Z=(Z_i)_{i=1}^d$, which are defined on the same probability space, and all $1\leq i,j\leq d$ we have that 
\[
|\mathbb E[X_iX_j]-\mathbb E[Z_iZ_j]|\leq \|X_i\|_{L^2}\|X_j-Z_j\|_{L^2}+
\|Z_j\|_{L^2}\|X_i-Z_i\|_{L^2}.
\]
It follows that
\[
|\Cov(X)-\Cov(Z)|\leq C_d\|X-Z\|_{L^2}(\|X\|_{L^2}+\|Z\|_{L^2}),
\]
where $C_d>0$ is a constant which depends only on $d$.
Let $X=\sum_{m=1}^{2^{n+1}}A_m$ and $Z=\sum_{m=1}^n\sum_{j=0}^{F(m)-1}X_{m,j}$. Then, by \cite[Proposition 4.1]{GO} we have that
\[
\|X-Z\|_{L^2}\leq C2^{n(1-\gamma)/2},
\]
for some constants  $C>0$ and $\gamma\in(0,1)$ which do not depend on $n$. Indeed, we observe that  the number of $A_m$'s appearing in $X-Z$ is at most of order $2^{(1-\gamma)n}$, $0<\gamma<1$. Furthermore, by \cite[Proposition 4.1]{GO}, the $L^2$-norms of $X$ and $Z$ are  of order at most $2^{n/2}$, and hence 
\[
|\Cov(X)-\Cov(Z)|\leq C'2^{n(1-\gamma)/2+n/2},
\]
where $C'>0$ does not depend on $n$. Using (\ref{CovCond}) we conclude that for every $v\in\mathbb R^d$ and a sufficiently large $n$, we have 
\begin{equation}\label{Approx0}
\Cov\bigg(\sum_{m=1}^n\sum_{j=0}^{F(m)-1}X_{m,j}\bigg)v\cdot v\geq \big(c_12^n-C'2^{n(1-\gamma)/2+n/2}\big)|v|^2\geq C_12^n|v|^2
\end{equation}
where $C_1>0$ is some constant. 
The corollary follows from (\ref{Approx0}) together with Proposition \ref{L2 Version}.
\end{proof}

We now recall the following proposition (see~\cite[Corollary 3]{Zai} or~\cite[Proposition 5.5]{GO}).
\begin{proposition}\label{Cor3}
Let $Y_0,\ldots ,Y_{b-1}$ be independent centered $\mathbb R^d$-valued random vectors. Let $q\geq2$ and set $M=\big(\sum_{j=0}^{b-1}\mathbb E|Y_j|^q\big)^{1/q}$. Assume that there exists a sequence $0=m_0<m_1<\ldots<m_s=b$ such that with $\zeta_k=Y_{m_k}+...+Y_{m_{k+1}-1}$ and $B_k=\Cov(\zeta_k)$, for every $v\in\mathbb R^d$ and $0\leq k<s$ we have that
\begin{equation}\label{BLOCKS}
100M^2|v|^2\leq B_kv\cdot v\leq 100CM^2|v|^2,
\end{equation}
where $C\geq1$ is some constant. Then, there exists a coupling between $(Y_0,\ldots,Y_{b-1})$ and a sequence of independent Gaussian random vectors $(S_0,\ldots ,S_{b-1})$ such that $\Cov(S_j)=\Cov(Y_j)$ for each $j\in \N$ and 
\begin{equation}\label{z}
\mathbb P\left(\max_{0\leq i\leq b-1}\left|\sum_{j=0}^{i}Y_j-S_j\right|\geq Mz\right)\leq C'z^{-q}+\exp(-C'z),
\end{equation}
for all $z\geq C'\log s$. Here, $C'$ is a positive constant which depends only of $C$,  $d$ and  $q$.
\end{proposition}

Now we can describe our next  modification of \cite{GO}.
In the proof of~\cite[Lemma 5.6]{GO}, Proposition~\ref{Cor3} was applied with $(Y_{n,0},\ldots ,Y_{n,F(n)-1})$. A key ingredient in the proof of this Lemma 5.6 was that the covariance matrix of each $Y_{n,j}$ is bounded from below (in the ordered set of semi-definite positive matrices) by an expression of the form $2^{(1-\beta)n}\Sigma^2(1+o(1))$, where $\Sigma$ is some positive definite matrix. In our circumstances we only have Corollary \ref{CovarianceCor}, and so we will be able to apply Proposition \ref{Cor3} successfully only with $\{Y_{m,j}\},\,(m,j)\prec(n,F(n)-1)$, as described in the following lemma.

\begin{lemma}\label{Lemma 5.6}
For every $n\in\mathbb N$, there exists a coupling between $\{Y_{m,j}\}$  and $\{S_{m,j}\}$, where $1\leq m\leq n$ and $1\leq j<F(m)$ and   $S_{n,j}$'s are independent centered Gaussian random variables with $Var(S_{n,j})=Var(Y_{n,j})$, such that
\begin{equation}\label{(5.11)}
\sum_{n} \mathbb P\left(\max_{(k,i)\prec (n,F(n)-1)}\left|\sum_{m=1}^k\sum_{j=0}^{i-1}Y_{k,j}-S_{k,j}\right|\geq 2^{\big((1-\beta)/2+\beta/p+\varepsilon/2\big)n}\right)<\infty.
\end{equation}
\end{lemma}
\begin{proof}
Take $q\in(2,p)$ and set 
\[
M=\Big(\sum_{m=1}^n\sum_{j=0}^{F(m)-1}\mathbb E|Y_{m,j}|^q\Big)^{\frac 1q}.
\]
By Proposition 4.1 in \cite{GO} we have 
\[
\|Y_{m,j}\|_{L^q}=\|X_{m,j}\|_{L^q}\leq C|I_{m,j}|^{\frac 12}\leq C'2^{(1-\beta)m/2},
\]
and therefore
\begin{equation}\label{M bound}
M\leq\sum_{m=1}^n\big(\sum_{j=0}^{F(m)-1}\mathbb E|Y_{m,j}|^q\big)^{\frac 1q}\leq 
C\sum_{m=1}^n (F(m))^{\frac 1q}2^{(1-\beta)m/2}\leq C 2^{\beta n/q}\cdot 2^{(1-\beta)n/2}.
\end{equation}
If we take $q$ sufficiently close to $p$, then $M^2$ is much smaller than $2^n$. On the other hand, by Corollary \ref{CovarianceCor} for any $v\in\mathbb R^d$ we have 
\[
\Cov\bigg(\sum_{m=1}^n\sum_{j=0}^{F(m)-1}Y_{m,j}\bigg)v\cdot v\geq c 2^n|v|^2.
\]
Observe next that for all $m$, $j$ and $v\in\mathbb R^d$,
\[
|\Cov(Y_{m,j})||v|^2\leq\|Y_{m,j}\|_{L^2}^2|v|^2\leq \|Y_{m,j}\|_{L^q}^2|v|^2\leq M^2|v|^2.
\]
Therefore, we can regroup $\{Y_{m,j}\},\,(m,j)\prec(n,F(n-1))$ so that (\ref{BLOCKS}) holds true with some $C'$ which does not depend on $n$ and with some $s$ (whose order in $n$ does not exceed $2^n$). Taking $z$ of the form $z=2^{\varepsilon n}$ in (\ref{z}) we obtain (\ref{(5.11)}). In the last argument,  we have used (\ref{M bound}), which insures that $M 2^{\varepsilon n}$ is much smaller than $2^{\big((1-\beta)/2+\beta/p+\varepsilon/2\big)n}$, when $q$ is close enough to $p$ and $\varepsilon$ is sufficiently small.
\end{proof} 

\begin{remark}
If a nonnegative random variable $X$ satisfies $P(X\geq Mz)\leq C'z^{-q}+\exp(-C'z)$ for all $z\geq C'\log s$, where $C',M>0$, $s\geq 2$ and $q>2$ then 
$$
\mathbb E[X^2]=\int_{0}^\infty\mathbb P(X^2\geq t)dt=2M^2\int_{0}^{\infty}z\mathbb P(X\geq Mz)\leq$$$$
2M^2\left(\int_{0}^{C'\log s}zdz+C'\int_{C'\log s}^\infty \big(z^{-q+1}+z\exp(-C'z)\big)dz\right)
\leq C''M^2(\log^2s+1)
$$
and so $\|X\|_2\leq CM\log s$ for some $C$ which depends only on $C'$ and $q$. Using this, it follows from the proof of Lemma \ref{Lemma 5.6} that for every $q\in(2,p)$ there is a constant $C_q>0$ so that for all $n\geq 2$,
\begin{equation}\label{Nest}
\left\|\max_{(k,i)\prec (n,F(n)-1)}\left|\sum_{m=1}^k\sum_{j=0}^{i-1}Y_{k,j}-S_{k,j}\right|\right\|_2\leq C_q 2^{\frac12n(2\beta/q+(1-\beta))}\log  n.
\end{equation}
\end{remark}
\textbf{Completing of the proof of Theorem \ref{Gthm}}.
The proof of Theorem 1.3 in  \cite{GO} is separated into six steps. All of these steps proceed exactly as in \cite{GO} except from Lemmas 5.6 and 5.7 there. In Lemma \ref{Lemma 5.6} we have proved a slightly weaker version of Lemma 5.6 in \cite{GO} which is clearly enough in order to obtain the desired approximation by sums of independent Gaussian random vectors. The purpose of Lemma 5.7 was to prescribe the covariances of the approximating Gaussians. 
 In  the first statement of Theorem~\ref{Gthm} we haven not claimed anything  about the variances of these Gaussians, and so we can skip in the corresponding part from~\cite{GO} and  complete the proof of (\ref{asip rate}) by taking $\beta=\frac{p}{2(p-1)}$. 
 
Next, let us  show that (\ref{Var est}) holds true. 
 Firstly, by applying Proposition \ref{L2 Version} with the finite sequence, we derive that
\[
\left\|\sum_{(n',j')\prec (n,j)}X_{n',j'}-Y_{n',j'}\right\|_{L^2}\leq C2^{\beta n/2}.
\]
Using also \eqref{Nest} with $q$ sufficiently close to $p$, and noting that $\beta/p+(1-\beta)=\beta/2$
 we obtain that there is a coupling of $\{X_{n',j'}\}$ with $\{S_{n',j'}\}$ so that
\begin{equation}\label{Up}
\left\|\sum_{(n',j')\prec (n,j)}X_{n',j'}-\sum_{(n',j')\prec (n,j)}S_{n',j'}\right\|_{L^2}\leq C2^{\beta n/2+\delta}.
\end{equation}
Of course, to obtain the above coupling we have also used the so-called Berkes–Philipp lemma (see \cite[Lemma A.1]{BP}, \cite[Lemma 3.1]{GO}).

Take $n\in\mathbb N$, and let $N_n$ be such that $2^{N_n}\leq n<2^{N_n+1}$. Furthermore, let $j_n$ be the largest index such that the left end point of $I_{N_n,j_n}$ is smaller than $n$. 
In the case when $n\in I_{N_n,j_n}$ we have
\begin{eqnarray*}
\sum_{i=1}^n A_i-\sum_{(n',j')\prec (N_n,j_n)}X_{n',j'}=\sum_{(n',j')\prec (N_n,j_n)}\,\sum_{i\in J_{n',j'}}A_i+\sum_{i\in J_{N_n,j_n}}A_i\\+\sum_{i=i_{N_n,j_n}}^{n}A_i=\sum_{i\leq n,\, i\in \mathcal J}A_i+\sum_{i=i_{N_n,j_n}}^{n}A_i:=I_1+I_2
\end{eqnarray*}
where $i_{n',j'}$ denotes the left end point of $I_{n',j'}$. Recall next that by \cite[(5.1)]{GO} the cardinality of $\mathcal J\cap[1,2^{N_n+1}]$ does not exceed $C2^{\varepsilon (N_n+1)}2^{\beta N_n}(\varepsilon N_n+2)$, which for our specific choice of $N_n$ does not exceed  $Cn^{\beta +3\varepsilon/2}$.
Using \cite[Lemma 5.9]{GO} with a sufficiently  small $\alpha$ we derive that
\[
\|I_1\|_{L^2}\leq C n^{\beta/2+\varepsilon}.
\]
On the other hand, applying \cite[Proposition 4.1]{GO} we obtain that
\[
\|I_2\|_{L^2}\leq C|I_{N_n,j_n}|^{\frac12}\leq C2^{N_n(1-\beta)/2}\leq Cn^{(1-\beta)/2}\leq Cn^{\beta/2}
\]
where we have used that for our specific choice of $\beta$ we have $(1-\beta)/2=\beta/2-\beta/p<\beta/2$.
We conclude that there exists a constant $C'>0$ so that for every $n\geq1$,
\begin{equation}\label{Del}
\left\|\sum_{j=1}^{n}A_j-\sum_{(n',j')\prec (N_n,j_n)}X_{n',j'}\right\|_{L^2}\leq C'n^{\beta/2+\varepsilon +1/p}.
\end{equation}
The proof of (\ref{Var est}) in the case when $n\in I_{N_n,j_n}$ is completed now by (\ref{Up}). 
The case when $n\not\in I_{N_n,j_n}$ is treated similarly. We first write
\[
\sum_{i=1}^n A_i-\sum_{(n',j')\prec (N_n,j_n)}X_{n',j'}=
\sum_{j\in \mathcal J, j\leq n}A_j+X_{N_n,j_n}:=I_1+I_2.
\]
Then the $L^2$-norms of the summands $I_1$ and $I_2$ are bounded exactly as in the case when $n\in I_{N_n,j_n}$, and the proof of \eqref{Var est}  in this case is also finalized by applying \eqref{Up}.

Finally, let us prove \eqref{Var est1} under the exponential decay of correlations assumptions \eqref{UniDec}. By replacing $A_n$ with $A_n\cdot v$, where $v$ is a unit vector, it is enough to prove  \eqref{Var est1} in the scalar case $d=1$. Recall first that the variances of $X_{n',j'}$ and $S_{n',j'}$ are equal. Therefore,
\[
\left\|\sum_{(n',j')\prec (n,j)}X_{n',j'}\right\|^2=
\left\|\sum_{(n',j')\prec (n,j)}S_{n',j'}\right\|^2+2\sum_{(n'_1,j'_1)\not= (n',j')}\text{Cov}(X_{n',j'},X_{n_1', j_1'})
\]
where  the double sum ranges over the pairs $(n'_1,j'_1)\not= (n',j')$ so that $(n'_1,j'_1),(n',j')\prec (n,j)$. Next, since the size of each $|I_{n,j}|$ is of order at most $2^n$ and the gaps between two different $I_{n,j}$'s is at least $2^{[n\beta]}$, using \eqref{UniDec} we conclude that there is an $a>0$ so that the above double sum is $O(2^{-an})$, and so
\begin{equation}\label{1-stA}
\left\|\sum_{(n',j')\prec (n,j)}X_{n',j'}\right\|^2=
\left\|\sum_{(n',j')\prec (n,j)}S_{n',j'}\right\|^2+O(2^{-an}).
\end{equation}  

Next, set $S_n=\sum_{(n',j')\prec (N_n,j_n)}(X_{n,j}+\check X_{n,j})$, where $\check X_{n',j'}=\sum_{k\in J_{n',j'}}A_k$.
  Set also $G_n=\sum_{(n',j')\prec (N_n,j_n)}X_{n',j'}$ and $\Delta_n=S_n-G_n$. Then 
\begin{equation}
\text{Var}(S_n)=\text{Var}(G_n)+\text{Var}(\Delta_n)+2\text{Cov}(\Delta_n,G_n).
\end{equation}
Now, by \eqref{Del} we have $\text{Var}(\Delta_n)=O(n^{\beta+\varepsilon})$. Moreover, as in the derivation of \eqref{1-stA} we have
\[
\text{Cov}(\Delta_n,G_n)=O(2^{-an})+\sum_{(n',j')\prec (N_n,j_n)}\text{Cov}(X_{n',j'},\check X_{n',j'}).
\]
 Because of \eqref{UniDec} we have that $\text{Cov}(X_{n',j'},\check X_{n',j'})$ is uniformly bounded in $(n',j')$. Therefore, 
$\text{Cov}(\Delta_n,G_n)=O(n^\beta)$. We conclude that 
\[
\text{Var}(S_n)=\text{Var}\left(\sum_{(n',j')\prec (N_n,j_n)}S_{n',j'}\right)+O(n^{\beta+\varepsilon}).
\]
Now, \eqref{UniDec} implies that 
\[
\text{Var}\left(\sum_{j=1}^n A_n\right)=\text{Var}(S_n)+O(n^{\beta+\varepsilon})
\]
and therefore \eqref{Var est1} follows by  the above estimates, noting that $\beta=\frac{p}{2(p-1)}$.
 \qed


\section{ASIP for random dynamical systems under abstract spectral conditions}\label{Random}

\subsection{Assumptions for a cocycle of transfer operators}
Let $(\Omega, \mathcal F, \mathbb P)$ be a probability space and assume that $\sigma \colon \Omega \to \Omega$ is an invertible measure-preserving transformation. Furthermore, we suppose that $\sigma$ is ergodic. 

Let
 $M$ be a compact Riemannian manifold (possibly with a boundary) equipped with the  Riemannian measure $m$ and assume that $\mathcal B=(\mathcal B, \| \cdot \|)$ is a Banach space whose elements are distributions on $M$ of order at most $q$, $q\in \N \cup \{0 \}$. More precisely, we require that there exists 
$C>0$ such that for each $h\in \mathcal B$ and $\phi \in C^q(M, \mathbb R)$, 
\begin{equation}\label{distr}
|h(\phi)| \le C\|h \| \cdot  |\phi |_{C^q}, 
\end{equation}
where $h(\phi)$ denotes the action of a distribution $h$ on a test function $\phi$. We say that $h\in \mathcal B$ is positive (and write $h\ge 0$) if $h(\phi) \ge 0$ for each test function $\phi \ge 0$. 

We assume that we have a family $(T_\omega)_{\omega \in \Omega}$ of maps $T_\omega \colon M \to M$ and the associated family 
 $(\mathcal L_\omega )_{\omega \in \Omega}$ of bounded operators on $\mathcal B$ such that:
\begin{itemize}
\item $\omega \mapsto \mathcal L_\omega$ is strongly measurable, i.e. $\omega \mapsto \mathcal L_\omega h$ is measurable for each $h\in \mathcal B$;
\item there is a norm $\| \cdot \|_w$ on $\mathcal B$ satisfying $\|\cdot \|_w \le \| \cdot \|$ and constants $a\in (0, 1)$ and $B_1, B_2>0$ such that 
\begin{equation}\label{LYineq}
\|\mathcal L_\omega^n h\| \le B_1 a^n \|h \| +B_2\|h\|_w \quad \text{for $\mathbb P$-a.e. $\omega \in \Omega$, $h\in \mathcal B$ and $n\in \N$,}
\end{equation}
where
\begin{equation}\label{compos}
\mathcal L_\omega^n=\mathcal L_{\sigma^{n-1} \omega} \circ \ldots \circ \mathcal L_{\sigma \omega} \circ \mathcal L_\omega;
\end{equation}
 \item for $\omega \in \Omega$, $h\in \mathcal B$ and a test function $\phi$, 
\begin{equation}\label{Sup1}
(\mathcal L_\omega h)(\phi)=h(\phi \circ T_\omega),
\end{equation}
where in the above formula it is implicitly assumed that the action $\phi \to h(\phi)$ induces an action $\phi \to h(\phi \circ T_\omega)$, for each $\omega$.
\end{itemize}

\begin{remark}
Observe that it follows from~\eqref{LYineq} (applied for $n=1$) that 
\begin{equation}\label{Sup}
\esssup_{\omega \in \Omega} \|\mathcal L_\omega \| <+\infty;
\end{equation}
\end{remark}

We consider $\xi \in \mathcal B^*$ given by 
\[
\langle \xi, h\rangle =h(1), \quad h\in \mathcal B.
\]
We suppose that there exist $D, \lambda >0$ such that 
\begin{equation}\label{Dec}
\|\mathcal L_\omega^n h \| \le De^{-\lambda n} \| h \|,
\end{equation}
for $\mathbb P$-a.e. $\omega \in \Omega$, $n\in \N$ and  $h \in \mathcal B_0:=\{h \in \mathcal B: \langle \xi, h \rangle =0\}$. By~\eqref{Sup1},
\begin{equation}\label{inv}
\langle \xi, \mathcal L_\omega h\rangle= \langle \xi, h\rangle, \quad \text{for $h\in \mathcal B$ and $\omega \in \Omega$.}
\end{equation}
The following result is a minor modification  of~\cite[Proposition 7]{DS}. We include the proof   for the sake of completeness.
\begin{proposition}\label{ho}
There exists a unique family $(h_\omega^0)_{\omega \in \Omega} \subset \mathcal B$ with the following properties:
\begin{itemize}
\item $\omega \mapsto h_\omega^0$ is measurable;
\item $h_\omega^0 \ge 0$ and $\langle \xi, h_\omega^0 \rangle=1$ for $\mathbb P$-a.e. $\omega \in \Omega$;
\item for $\mathbb P$-a.e. $\omega \in \Omega$, $\mathcal L_\omega h_\omega^0=h_{\sigma \omega}^0$;
\item we have that 
\begin{equation}\label{20f}
\esssup_{\omega \in \Omega} \| h_\omega^0 \|<+\infty. 
\end{equation}
\end{itemize}
\end{proposition}

\begin{proof}
Let $\mathcal Y$ denote the space of all measurable  $v\colon \Omega \to \mathcal B$ such that 
\[
\|v \|_\infty:=\esssup_{\omega \in \Omega} \| v(\omega) \|<+\infty.
\]
Then, $(\mathcal Y, \| \cdot \|_\infty)$ is a Banach space. By $\mathcal Z$ we denote the subset of $\mathcal Y$ consisting of all $v\in \mathcal Y$ such that 
\[
v(\omega) \ge 0 \quad \text{and} \quad \langle \xi, v(\omega)\rangle =1, \quad \text{for $\mathbb P$-a.e. $\omega \in \Omega$.}
\]
Using~\eqref{distr}, it is easy to show that $\mathcal Z$ is closed.  We define $\mathbb L \colon \mathcal Z \to \mathcal Z$ by
\[
(\mathbb L v) (\omega)=\mathcal L_{\sigma^{-1} \omega} v(\sigma^{-1} \omega), \quad \omega \in \Omega, \ v\in \mathcal Z.
\]
By~\eqref{Sup1}, \eqref{Sup},  \eqref{inv} and recalling that $\omega \mapsto \mathcal L_\omega$ is strongly measurable, we have that $\mathbb L$ is well-defined. Moreover, choosing $N$ so that $De^{-\lambda N} <1$, it follows from~\eqref{Dec} that $\mathbb L^N$ is a contraction.
Hence, $\mathbb L$ has a unique fixed point $v \in \mathcal Z$. Set $h_\omega^0=v(\omega)$ for $\omega \in \Omega$. Clearly, the family $(h_\omega^0)_{\omega \in \Omega}$ satisfies the desired properties. Conversely, it is obvious that each family satisfying the properties as in the statement of the proposition gives rise to a fixed point of $\mathbb L$, which is unique. 
\end{proof}

Let
\begin{equation}\label{Q}
Q_\omega h=\langle \xi, h\rangle h_\omega^0=h(1)h_\omega^0, \quad \text{for $h\in \mathcal B$ and $\omega \in \Omega$.}
\end{equation}
Following~\eqref{compos}, we set 
\[
Q_\omega^n=Q_{\sigma^{n-1} \omega} \circ \ldots \circ Q_{\sigma \omega} \circ Q_\omega, \quad \omega \in \Omega, \  n\in \N.
\]
\begin{corollary}\label{Kor}
There exists a constant $D'>0$ such that 
\[
\| \mathcal L_\omega^n-Q_\omega^n \| \le D'e^{-\lambda n}, \quad \text{for $\mathbb P$-a.e. $\omega \in \Omega$ and $n \in \N$.}
\]
\end{corollary}

\begin{proof}
Observe that 
\[
(\mathcal L_\omega^n-Q_\omega^n)h=\mathcal L_\omega^n h-\langle \xi, h\rangle h_{\sigma^n \omega}=\mathcal L_\omega^n (h-\langle \xi, h\rangle h_\omega^0),
\]
and $h-\langle \xi, h\rangle h_\omega^0 \in \mathcal B_0$. Hence, the desired conclusion follows from~\eqref{Dec} and~\eqref{20f}.
\end{proof}

From now on, we assume that 
\begin{equation}\label{pm}
h_\omega^0 \quad \text{is a probability measure on $M$ for $\mathbb P$-a.e. $\omega \in \Omega$.}
\end{equation}

\begin{remark}\label{rpm}
In concrete applications, in order to fulfill the above requirement, it will be sufficient to show that $h_\omega^0$ is a Radon measure on $M$. Then, the second assertion of Proposition~\ref{ho} will imply that $h_\omega^0$ is a probability measure. 
\end{remark}

\subsection{Assumptions for a space of observables}\label{OS}
Let $\mathcal O=(\mathcal O, \| \cdot \|_o)$ be  a Banach space consisting of measurable functions $g \colon M \to \mathbb C$  and has the following  properties:
\begin{itemize}
\item constant functions belong to $\mathcal O$;
\item $\phi \in \mathcal O$ if and only if $\overline{\phi} \in \mathcal O$;
\item for $g_1, g_2 \in \mathcal O$, $g_1g_2 \in \mathcal O$. Moreover, 
\begin{equation}\label{prod2}
\|g_1 g_2 \|_o \le K_o \|g_1\|_o \cdot \|g_2\|_o, 
\end{equation}
for some constant $K_o >0$ which is independent on $g_1, g_2$.
\end{itemize}
\begin{remark}
 Without any loss of generality, we may assume (and from now on we will) that $\|1\|_o=1$. 
\end{remark}
 In addition, we suppose that there exists a bilinear operator $\cdot \colon \mathcal O \times \mathcal B \to \mathcal B$ with the property that 
there exists $K_o'>0$ (independent on $g$ and $h$) such that:
\begin{equation}\label{prod}
\| g \cdot h \| \le K_o'\|g \|_o \cdot \|h \| \quad \text{for every $g \in \mathcal O$ and $h\in \mathcal B$,}
\end{equation}
Moreover, we  assume that the action of $g\cdot h$ as a distribution is given by
\begin{equation}\label{product}
(g\cdot h)(\phi)=h(g\phi), \quad \text{for each test function $\phi$.}
\end{equation}
\begin{remark}
Clearly, we may take $K_o=K_o'$.
\end{remark}
Finally, we suppose that there is a continuous function $D \colon [0, +\infty) \times [0, +\infty) \to (0, +\infty)$ with the following properties:
\begin{itemize}
\item $D$ is increasing in the second variable and $\lim_{\theta \to 0} D(\theta, x)=0$ for each $x\ge 0$;
\item for each $\theta \in \mathbb C$ and $g\in \mathcal O$, $e^{\theta g} \in \mathcal O$ and 
\begin{equation}\label{exp}
\| e^{\theta g} -1 \|_o \le D(|\theta|, \|g \|_o).
\end{equation}
\end{itemize}

\begin{remark}
In all of our examples, the last condition will be a simple consequence of the mean-value theorem.
\end{remark}

\subsection{Good observables and  auxiliary results}
Take now $d\in \N$. We say that $g\colon \Omega \times X \to \mathbb R^d$, $g=(g^1, \ldots, g^d)$ is a \emph{good observable} if the following holds:
\begin{itemize}
\item $g$ is measurable;
\item $g_\omega^i:=g^i(\omega, \cdot) \in \mathcal O$ for $\omega \in \Omega$ and $1\le i \le d$;
\item \begin{equation}\label{Obs}
\| g\|_\infty:= \max_{1\le i \le d} \esssup_{\omega \in \Omega}  \|g_\omega^i \|_o <+\infty. 
\end{equation}
\end{itemize}

Take now a good observable $g\colon \Omega \times X \to \mathbb R^d$.
For $\theta \in \mathbb C^d$, we define a linear operator $\mathcal L_\omega^\theta \colon \mathcal B \to \mathcal B$ by 
\begin{equation}\label{twist}
\mathcal L_\omega^\theta h=\mathcal L_\omega (e^{\theta \cdot g_\omega} h), \quad h \in \mathcal B.
\end{equation}

\begin{proposition}\label{2115}
There exists a continuous function $L\colon \mathbb C \to (0, +\infty)$ such that 
\begin{equation}\label{z1}
\esssup_{\omega \in \Omega} \|\mathcal L_\omega^\theta \| \le L(\theta),
\end{equation}
for every $\theta \in \mathbb C^d$.
\end{proposition}

\begin{proof}
By~\eqref{prod2} and~\eqref{exp}, we have that 
\[
\|e^{\theta \cdot g_\omega} \|_o =\| \prod_{i=1}^d e^{\theta_i g_\omega^i} \|_o \le K_o^d \prod_{i=1}^d \|e^{\theta_i g_\omega^i} \|_0 \le K_o^d \prod_{i=1}^d  (D(|\theta_i |, \|g\|_\infty)+1).
\]
Hence, \eqref{prod} and~\eqref{twist}  imply that~\eqref{z1} holds with $L\colon \mathbb C \to (0, +\infty)$ given by
\[
L(\theta)=\big (\esssup_{\omega \in \Omega} \|\mathcal L_\omega \big  \| ) K_o^{d+1}  \prod_{i=1}^d  (D(|\theta_i |, \|g\|_\infty)+1).
\]
Clearly, $L$ is continuous.
\end{proof}

\begin{proposition}
There exists a continuous function $\tilde L \colon \mathbb C \to (0, +\infty)$ such that $\lim_{\theta \to 0} \tilde L(\theta)=0$ and
\begin{equation}\label{z2}
\esssup_{\omega \in \Omega} \|\mathcal L_\omega^\theta -\mathcal L_\omega \| \le \tilde L(\theta),
\end{equation}
for every $\theta \in \mathbb C^d$.
\end{proposition}

\begin{proof}
Observe that 
\[
(\mathcal L_\omega^\theta -\mathcal L_\omega)h=\mathcal L_\omega ((e^{\theta \cdot g}-1)\cdot h).
\]
It follows from~\eqref{prod2}, \eqref{exp} and the triangle inequality that 
\[
\begin{split}
\| e^{\theta \cdot g}-1 \|_o &= \| \prod_{i=1}^d e^{\theta_i g_\omega^i} -1\|_o  \\
& \le \sum_{i=1}^d \| \prod_{j=1}^i e^{\theta_j g_\omega^j}-\prod_{j=1}^{i-1} e^{\theta_j g_\omega^j} \|_o \\
&\le K_o \sum_{i=1}^d \|\prod_{j=1}^{i-1} e^{\theta_j g_\omega^j} \|_o \cdot \|e^{\theta_i g_\omega^i} -1\|_o \\
&\le  \sum_{i=1}^d K_o^{i} D(|\theta_i|, \|g\|_\infty) \prod_{j=1}^{i-1} (D(|\theta_j |, \|g\|_\infty)+1),
\end{split}
\]
and thus~\eqref{z2} holds with $\tilde L\colon \mathbb C \to (0, +\infty)$ given by
\[
\tilde L(\theta)=\big (\esssup_{\omega \in \Omega} \|\mathcal L_\omega \big  \| ) \sum_{i=1}^d K_o^{i+1} D(|\theta_i|, \|g\|_\infty) \prod_{j=1}^{i-1} (D(|\theta_j |, \|g\|_\infty)+1).
\]
Obviously, $\tilde L$ is continuous and $\lim_{\theta \to 0} \tilde L(\theta)=0$.
\end{proof}
As in~\eqref{compos}, we set
\[
\mathcal L_\omega^{\theta, n}=\mathcal L_{\sigma^{n-1} \omega}^\theta \circ \ldots \circ \mathcal L_\omega^\theta, \quad \text{for $\omega \in \Omega$, $n\in \N$ and $\theta \in \mathbb C$.}
\]
Choose $N\in \N$ such that $\gamma:=B_1a^N<1$, where $B_1$ and $a$ are as in~\eqref{LYineq}. The proof of the following result will be obtained by arguing as in the proof of~\cite[Proposition 4.4]{DFGTV}. 
\begin{proposition}\label{2116}
There exists $\tilde \gamma \in (0, 1)$ such that for all $\theta \in \mathbb C^d$ sufficiently close to $0$,
\begin{equation}\label{twistLY}
\|\mathcal L_\omega^{\theta, N} h\| \le \tilde \gamma \|h \| +B_2\|h\|_w, \quad \text{for $\mathbb P$-a.e. $\omega \in \Omega$ and  $h\in \mathcal B$.}
\end{equation}
\end{proposition}

\begin{proof}
We have
\begin{equation}\label{uu}
 \begin{split}
  \lVert \mathcal L_\omega^{\theta, N} h\rVert  &\le \lVert  \mathcal L_\omega^{ N} h\rVert+\lVert \mathcal L_\omega^{\theta, N}-\mathcal L_\omega^{N}\rVert \cdot \lVert h\rVert \\
  &\le
  \gamma \lVert h\rVert+B_2\lVert h\rVert_w+\lVert \mathcal L_\omega^{\theta, N}-\mathcal L_\omega^{N}\rVert \cdot \lVert h\rVert.
  \end{split}
 \end{equation}
 On the other hand, 
\begin{equation}\label{sum}
 \mathcal L_\omega^{\theta, N}-\mathcal L_\omega^{N}=\sum_{j=0}^{N-1} \mathcal L_{\sigma^{N-j} \omega}^{\theta, j}(\mathcal L_{\sigma^{N-1-j} \omega}^\theta -\mathcal L_{\sigma^{N-1-j} \omega})\mathcal L_\omega^{N-1-j}.
\end{equation}
Let $L\colon \mathbb C \to (0,+\infty)$ be given by Proposition~\ref{2115}. Since $L$ is continuous, we have that 
\begin{equation}\label{M}
M:=\sup_{|\theta | \le 1} L(\theta)<+\infty.
\end{equation}
Hence, 
\[
\| \mathcal L_{\sigma^{N-j} \omega}^{\theta, j} \| \le M^j \quad \text{and} \quad \|\mathcal L_\omega^{N-1-j}\| \le M^{N-1-j},
\]
for $\mathbb P$-a.e. $\omega \in \Omega$, $|\theta| \le 1$ and $0\le j \le N-1$. By~\eqref{z2} and~\eqref{sum}, we have that 
\[
\|\mathcal L_\omega^{\theta, N}-\mathcal L_\omega^{N} \| \le NM^{N-1}\tilde L(\theta), \quad \text{for $\mathbb P$-a.e. $\omega \in \Omega$ and $|\theta| \le 1$.}
\]
The above estimate together with~\eqref{uu} readily implies the desired conclusion.
\end{proof}

\begin{proposition}
Assume that there exists $B_3>0$ such that 
\begin{equation}\label{weakLY}
\| \mathcal L_\omega^{it, n} h \|_w \le B_3 \| h \|_w, \quad  \text{for $\mathbb P$-a.e. $\omega \in \Omega$, $h\in \mathcal B$, $t\in \mathbb R^d$, $|t| \le 1$ and $n\in \N$.}
\end{equation}
Then there exists $\rho, C>0$ such that 
\begin{equation}\label{it}
\|\mathcal L_\omega^{it, n} \| \le C, \quad \text{for $\mathbb P$-a.e. $\omega \in \Omega$, $n\in \N$ and $t \in \mathbb R^d$, $|t| \le \rho$.}
\end{equation}
\end{proposition}

\begin{proof}
 Let $\rho \in (0, 1)$ be such that~\eqref{twistLY} holds when $|\theta| \le \rho$. By iterating~\eqref{twistLY} and using~\eqref{weakLY}, we find that there exists $C_1>0$ such that
\begin{equation}\label{jj}
\|\mathcal L_\omega^{it, nN}\| \le C_1, \quad \text{for $\mathbb P$-a.e. $\omega \in \Omega$, $n\in \N$ and $t \in \mathbb R^d$, $|t| \le \rho$.}
\end{equation}
The desired conclusion follows readily from~\eqref{z1}, \eqref{M} and~\eqref{jj}.
\end{proof}

\subsection{Almost sure invariance principle}
By $\mathbb E_\omega(\phi)$ we will denote the expectation of $\phi$ with respect to $h_\omega^0$ (see~\eqref{pm}). 

Throughout this subsection, we take a good observable $g$ such that~\eqref{weakLY} holds (for some $B_3>0$). Set
\[
T_\omega^n=T_{\sigma^{n-1} \omega} \circ \ldots \circ T_\omega
\]
and 
\[
S_n g(\omega, \cdot)= \sum_{i=0}^{n-1} g(\sigma^i \omega, T_\omega^i(\cdot)),
\]
for $\omega \in \Omega$ and $n\in \N$.

\begin{lemma}
\label{dualtwist2} 
For  $\omega\in\Omega$, $h\in \mathcal{B}$, $n\in \N$ and a test function $\phi$, we have that 
\[
  (\mathcal{L}_\omega^{\theta, n}h)(\phi)=h(e^{\theta \cdot S_ng(\omega, \cdot)}(\phi \circ T_\omega^{n})).
\]
\end{lemma}

\begin{proof}
The desired conclusion follows from~\eqref{Sup1} and~\eqref{product} (by using induction on $n$).
\end{proof}

\begin{lemma}\label{631l}
Let $g\colon \Omega \times M \to \mathbb R^d$ be a good observable. 
For $\mathbb P$-a.e. $\omega \in \Omega$, 
there exist $C,c,\rho>0$ such that for any $n,m>0$, $b_1<b_2< \ldots <b_{n+m+1}$, $k>0$ and $t_1,\ldots ,t_{n+m}\in\mathbb R^d$ with $|t_j|\leq \rho$, we have that
\begin{eqnarray*}
\Big|\mathbb E_\omega\big(e^{i\sum_{j=1}^nt_j \cdot (\sum_{\ell=b_j}^{b_{j+1}-1}A_\ell)+i\sum_{j=n+1}^{n+m}t_j \cdot (\sum_{\ell=b_j+k}^{b_{j+1}+k-1}A_\ell)}\big)\\
-\mathbb E_\omega\big(e^{i\sum_{j=1}^nt_j \cdot (\sum_{\ell=b_j}^{b_{j+1}-1}A_\ell)}\big)\cdot\mathbb E_\omega\big(e^{i\sum_{j=n+1}^{n+m}t_j \cdot (\sum_{\ell=b_j+k}^{b_{j+1}+k-1}A_\ell)}\big)\Big|\\\leq C(1+\max|b_{j+1}-b_j|)^{C(n+m)}e^{-ck},
\end{eqnarray*}
\end{lemma}
where \[A_\ell:=g_{\sigma^\ell \omega} \circ T_\omega^\ell, \quad \ell\in \N\cup\{0\}. \]
\begin{proof}
For $\omega \in \Omega$, let $Q_\omega$ be given by~\eqref{Q}.
By applying Lemma~\ref{dualtwist2}, one can verify that
\[
\begin{split}
& \mathbb E_\omega\big(e^{i\sum_{j=1}^nt_j \cdot (\sum_{\ell=b_j}^{b_{j+1}-1}A_\ell)+i\sum_{j=n+1}^{n+m}t_j \cdot (\sum_{\ell=b_j+k}^{b_{j+1}+k-1}A_\ell)}\big) \\
&=          \mathcal L_{\sigma^{b_{n+m}+k}\omega}^{it_{n+m}, b_{n+m+1}-b_{n+m}}                       \ldots   \mathcal L_{\sigma^{b_{n+1}+k}\omega}^{it_{n+1}, b_{n+2}-b_{n+1}}   \mathcal L_{\sigma^{b_{n+1}} \omega}^{k}  \mathcal L_{\sigma^{b_n}\omega}^{it_n, b_{n+1}-b_n}   \ldots  \mathcal L_{\sigma^{b_1} \omega}^{it_1, b_2-b_1}\mathcal L_\omega^{b_1}  
h_\omega^0 (1) \\
&=\mathcal L_{\sigma^{b_{n+m}+k}\omega}^{it_{n+m}, b_{n+m+1}-b_{n+m}}                       \ldots   \mathcal L_{\sigma^{b_{n+1}+k}\omega}^{it_{n+1}, b_{n+2}-b_{n+1}} (\mathcal L_{\sigma^{b_{n+1}} \omega}^{k}-Q_{\sigma^{b_{n+1}} \omega}^{k}) \mathcal L_{\sigma^{b_n}\omega}^{it_n, b_{n+1}-b_n}   \ldots  \mathcal L_\omega^{b_1}  
 h_\omega^0 (1) \\
&\phantom{=}+\mathcal L_{\sigma^{b_{n+m}+k}\omega}^{it_{n+m}, b_{n+m+1}-b_{n+m}}                       \ldots   \mathcal L_{\sigma^{b_{n+1}+k}\omega}^{it_{n+1}, b_{n+2}-b_{n+1}} Q_{\sigma^{b_{n+1}} \omega}^{k} \mathcal L_{\sigma^{b_n}\omega}^{it_n, b_{n+1}-b_n}   \ldots  \mathcal L_\omega^{b_1} 
 h_\omega^0 (1) .\\
\end{split}
\]
It follows from~\eqref{distr}, \eqref{20f}, Corollary~\ref{Kor} and~\eqref{it} that 
\[
\begin{split}
&\bigg{\lvert}\mathcal L_{\sigma^{b_{n+m}+k}\omega}^{it_{n+m}, b_{n+m+1}-b_{n+m}}                       \ldots   \mathcal L_{\sigma^{b_{n+1}+k}\omega}^{it_{n+1}, b_{n+2}-b_{n+1}} (\mathcal L_{\sigma^{b_{n+1}} \omega}^{k}-Q_{\sigma^{b_{n+1}} \omega}^{k}) \mathcal L_{\sigma^{b_n}\omega}^{it_n, b_{n+1}-b_n}   \ldots  \mathcal L_\omega^{b_1}  
 h_\omega^0 (1)  \bigg{\rvert} \\
&\le C'e^{-\lambda k}C^{n+m},
\end{split}
\]
for some constant $C'>0$.
Moreover, Lemma~\ref{dualtwist2} implies that
\[
\begin{split}
&\mathcal L_{\sigma^{b_{n+m}+k}\omega}^{it_{n+m}, b_{n+m+1}-b_{n+m}}                       \ldots   \mathcal L_{\sigma^{b_{n+1}+k}\omega}^{it_{n+1}, b_{n+2}-b_{n+1}} Q_{\sigma^{b_{n+1}} \omega}^{k} \mathcal L_{\sigma^{b_n}\omega}^{it_n, b_{n+1}-b_n}   \ldots  \mathcal L_\omega^{b_1} 
 h_\omega^0 (1) \\
&=\mathbb E_\omega\big(e^{i\sum_{j=1}^nt_j \cdot (\sum_{\ell=b_j}^{b_{j+1}-1}A_\ell)}\big)\cdot\mathbb E_\omega\big(e^{i\sum_{j=n+1}^{n+m}t_j \cdot (\sum_{\ell=b_j+k}^{b_{j+1}+k-1}A_\ell)}\big),
\end{split}
\]
and the conclusion of the lemma follows. 
\end{proof}

In order to obtain our main result, we introduce an additional assumption. For $h\in \mathcal B$, we assume that the action $\phi \to h(\phi)$ induces an action
$\phi \to h(\phi)$ on $\mathcal O$ and that there exists $L>0$ such that 
\begin{equation}\label{nc}
|h(\phi)| \le L \|h \| \cdot \|\phi \|_o, \quad \text{for $\phi \in \mathcal O$ and $h\in \mathcal B$.}
\end{equation}
\begin{remark}
In principle, instead of assuming \eqref{nc} we could have assumed from the start that $\mathcal B$ is a space of distributions acting on $\mathcal O$. In this case \eqref{distr} and \eqref{nc} coincide. For this to make sense, we would need to require that $\mathcal B$ includes all finite measures, which holds when $\|g\|_{o}\geq \sup |g|$, for all $g\in\mathcal O$. For instance, in the setting of Subsection~\ref{LY},  $\mathcal B$ will be the space of  functions $h$ on $M$ with bounded variation, identified as linear functionals given by $h(\phi)=\int_M h\phi \, dm$. Note that the  $BV$ norm is  larger than the supremum norm. In this case it will make no difference if we consider $h\in BV$ as a distribution on $C^0(M,\mathbb R)$ or on the space of $BV$ functions. 

The reason we have decided not to present our results in this way and to assume \eqref{nc} is that it is less standard to consider an abstract class of test functions.
\end{remark}

\begin{lemma}\label{corr}
There exists $C>0$ such that 
\[
|h_\omega^0(\phi (\psi \circ T_\omega^n))| \le Ce^{-\lambda n}\| \phi \|_o \cdot \|\psi \|_o, 
\]
for $\mathbb P$-a.e. $\omega \in \Omega$, $n\in \N$ and $\phi, \psi \in \mathcal O$ such that $h_\omega^0(\phi)=0$.
\end{lemma}

\begin{proof}
We have that 
\[
h_\omega^0(\phi (\psi \circ T_\omega^n))=\mathcal L_\omega^n (\phi \cdot h_\omega^0)(\psi).
\]
Since $\phi \cdot h_\omega^0 \in \mathcal B_0$, the desired conclusion follows from~\eqref{Dec}, \eqref{20f}, \eqref{prod} and~\eqref{nc}.
\end{proof}

We consider the skew-product transformation $\tau \colon \Omega \times M \to \Omega \times M$ given by
\begin{equation}\label{spt}
\tau (\omega, x)=(\sigma \omega, T_\omega (x)), \quad (\omega, x) \in  \Omega \times M.
\end{equation}
Furthermore, we consider the probability measure $\mu$ on $\Omega \times M$ given by
\[
\mu(A\times B)=\int_\Omega h_\omega^0(B)\, d\mathbb P(\omega), \quad \text{for $A\in \mathcal F$ and $B\subset M$ Borel.}
\]
It follows easily from the third assertion of Proposition~\ref{ho} that $\mu$ is invariant for $\tau$.
Let us also assume that $\mu$ is ergodic. 
Next, we need the following result.
\begin{proposition}\label{VarProp}
Let $g\colon \Omega \times M \to \mathbb R$ be a good observable such that $g\in L^2(\Omega \times M, \mu)$ and $h_\omega^0(g_\omega)=0$ for $\mathbb P$-a.e. $\omega \in \Omega$. Then, there exists $\Sigma^2\ge 0$ such that 
\[
\lim_{n\to \infty} \frac  1 n \mathbb E_\omega \bigg (\sum_{k=0}^{n-1} g_{\sigma^k \omega}\circ T_\omega^k \bigg )^2=\Sigma^2, \quad \text{for $\mathbb P$-a.e. $\omega \in \Omega$.}
\]
Moreover,
\begin{equation}\label{SF}
\Sigma^2=\int_{\Omega \times M} g^2\, d\mu+2\sum_{n=1}^\infty \int_{\Omega \times M}  g(g\circ \tau^n)\, d\mu.
\end{equation}
Finally, $\Sigma^2=0$ if and only if there exists $q\in L^2(\Omega \times M, \mu)$ such that 
\begin{equation}\label{Cob0}
g=q-q\circ \tau, \quad \text{$\mu$-a.s.}
\end{equation}
\end{proposition}

\begin{proof}
The proof of the following result can be obtained by arguing exactly as in the proofs of~\cite[Lemma 12]{DFGTV1} and~\cite[Proposition 3]{DFGTV1}. For reader's convenience we provide here some of the details.
 Note first that
 \[
 \begin{split}
  \mathbb E_\omega \bigg{(}\sum_{k=0}^{n-1} g_{\sigma^k \omega} \circ T_\omega^k \bigg{)}^2 &=
  \sum_{k=0}^{n-1} \mathbb E_\omega (g_{\sigma^k \omega}^2\circ T_\omega^k )+2\sum_{0\le i<j\le n-1}\mathbb E_\omega ((g_{\sigma^i \omega} \circ T_\omega^i)
  (g_{\sigma^j \omega} \circ T_\omega^j)) \\
&=\sum_{k=0}^{n-1} \mathbb E_\omega (g_{\sigma^k \omega}^2\circ T_\omega^k )
+2\sum_{i=0}^{n-1}\sum_{j=i+1}^{n-1}\mathbb E_{\sigma^i \omega}(g_{\sigma^i \omega}( g_{\sigma^j \omega}\circ T_{\sigma^i \omega}^{j-i})).
  \end{split}
 \]
Set $G(\omega)=\mathbb E_\omega (g_\omega^2)$, $\omega \in \Omega$. By applying Birkhoff's ergodic theorem for $G$ over the ergodic measure-preserving system $(\Omega, \mathcal F, \mathbb P, \sigma)$, we find that
\[
\begin{split}
\lim_{n\to \infty} \frac 1 n  \sum_{k=0}^{n-1}  \mathbb E_\omega &(g_{\sigma^k \omega}^2\circ T_\omega^k ) =\lim_{n\to \infty} \frac 1n \sum_{k=0}^{n-1} G(\sigma^k \omega)=\int_\Omega G(\omega)\, d\mathbb P(\omega) \\
&=\int_{\Omega} \int_M g(\omega, x)^2 \, dh_\omega^0(x) \, d\mathbb P(\omega)
=\int_{\Omega \times M}g(\omega, x)^2\, d\mu (\omega, x),
\end{split}
\]
for $\mathbb P$-a.e. $\omega \in \Omega$.
Furthermore, set
\[
\Psi(\omega)=\sum_{n=1}^\infty \int_M  g(\omega, x)g(\tau^n (\omega, x))\, dh_\omega^0(x)=\sum_{n=1}^\infty\mathcal L_\omega^n (g_\omega \cdot h_\omega^0)(g_{\sigma^n\omega}).
\]
By \eqref{Dec}, \eqref{20f}, \eqref{Obs} and~\eqref{nc},  there is a constant $c>0$ so that
\[
\lvert \Psi(\omega)\rvert  \leq c, \quad \text{for $\mathbb P$-a.e. $\omega \in \Omega$.}
\]
 In particular, $\Psi \in L^1(\Omega)$ and thus it follows again from Birkhoff's ergodic theorem   that
\begin{equation}\label{DD}
\lim_{n\to \infty}\frac{1}{n} \sum_{i=0}^{n-1}\Psi(\sigma^i \omega)=\int_\Omega \Psi(\omega)\, d\mathbb P(\omega)=\sum_{n=1}^\infty \int_{ \Omega \times X}g (\omega, x)g (\tau^n(\omega, x))\, d\mu (\omega, x),
\end{equation}
for $\mathbb P$-a.e. $\omega \in \Omega$. In order to complete the proof of the existence of $\Sigma^2$, it is enough  to show that
\begin{equation}\label{cv}
\lim_{n\to \infty} \frac 1 n \bigg{(}\sum_{i=0}^{n-1}\sum_{j=i+1}^{n-1}\mathbb E_{\sigma^i \omega}(g_{\sigma^i \omega}(g_{\sigma^j \omega}\circ T_{\sigma^i \omega}^{j-i}))-\sum_{i=0}^{n-1}\Psi(\sigma^i \omega)\bigg{)}=0,
\end{equation}
for $\mathbb P$-a.e. $\omega \in \Omega$. However, \eqref{cv} can be obtained by arguing exactly as in the proof of~\cite[Proposition 3]{DFGTV1}. We conclude that the first assertion of the proposition holds. 

To prove the characterization for the positivity of $\Sigma^2$, let $X_n=g\circ\tau^n$, considered as a random variable with respect to the measure $\mu$. Then $\{X_n\}$ is a stationary sequence satisfying $\sum_n (n+1) \mathbb |E_\mu[X_nX_0]|<\infty$ (using Lemma \ref{corr}).
Thus,  using classical results for stationary sequences  (see \cite{IL}) we  conclude from \eqref{SF}  that
$$
\Sigma^2=\lim_{n\to\infty}\frac1n\text{Var}\left(\sum_{j=0}^{n-1}g\circ \tau^j\right).
$$
Moreover, \eqref{Cob0} is equivalent to $\Sigma^2=0$.
\end{proof}

\begin{proposition}\label{CovProp}
Let $g\colon \Omega \times M \to \mathbb R^d$ be a good observable such that $g\in L^2(\Omega \times M, \mu)$ and $h_\omega^0(g_\omega)=0$ for $\mathbb P$-a.e. $\omega \in \Omega$. Then, 
there exists a positive semi-definite $d\times d$ matrix $\Sigma^2$ such that for $\mathbb P$-a.e. $\omega \in \Omega$ we have that 
\[
\lim_{n\to\infty}\frac 1n \mathbb E_\omega\big(S_n g(\omega, \cdot)\big)^2=\Sigma^2.
\]
 Moreover, $\Sigma^2$ is not positive definite if and only if there exist  $v\in\mathbb R^d\setminus \{0\}$ and an $\mathbb R$-valued function $r\in L^2(\Omega \times M, \mu)$ such that
\begin{equation}\label{Cob}
 v\cdot g=r-r\circ \tau,\,\,\,\mu-\text{a.s.}
\end{equation} 
\end{proposition}

\begin{proof}
Let $v\in\mathbb R^d$ and consider the real valued function $g_{v}=v\cdot g$.  By applying Proposition~\ref{VarProp},  we obtain that there exists $\Sigma_v^2\ge 0$ such that
\begin{equation}\label{bbx}
\lim_{n\to \infty} \frac 1 n \mathbb E_\omega (S_n g_v(\omega, \cdot))^2=\Sigma_v^2, \quad \text{for $\mathbb P$-a.e. $\omega \in \Omega$.}
\end{equation}
Moreover, $\Sigma_v^2=0$ if and only if there exists $r\in L^2(\Omega \times M, \mu)$ such that
\begin{equation}\label{Cob}
g_v=r-r\circ \tau.
\end{equation} 

For $1\le i, j\le d$,  we claim that there exists a real number $\Sigma^2_{i,j}$ so that $\mathbb P$-a.e. we have that
\begin{equation}\label{Sig ij}
\lim_{n\to\infty}\frac1n\mathbb E_\omega(S_n g^i(\omega,\cdot)\cdot S_n g^j(\omega,\cdot))=\Sigma^2_{i,j}.
\end{equation}
Clearly, it follows from~\eqref{bbx} that  $\Sigma_{i,j}^2=\big(\Sigma_{e_i+e_j}^2-\Sigma_{e_i}^2-\Sigma_{e_j}^2\big)/2$ satisfies \eqref{Sig ij}, where $e_i$ denotes  the standard $i$-th unit vector. 
 The resulting matrix $\Sigma^2=(\Sigma^2_{i,j})$ is positive semi-definite.  Indeed, it is easy to verify that $\Sigma^2 v\cdot v=\Sigma^2_v$ for $v\in \mathbb R^d$. From this we also see that $\Sigma^2$ is not positive definite if and only if there exists $v\in \mathbb R^d$, $v\neq 0$ such that 
$\Sigma^2_v=0$. However,  it follows from  the previous paragraph that this happens  if and only if $v\cdot g=g_v$ can be written in the form~\eqref{Cob}. The proof of the proposition is completed. 

\end{proof}

\begin{theorem}\label{Th}
Let $g\colon \Omega \times M \to \mathbb R^d$ be a good observable such that $g\in L^\infty (\Omega \times M, \mu)$ and  $h_\omega^0(g_\omega)=0$ for $\mathbb P$-a.e. $\omega \in \Omega$. Furthermore, suppose that $\Sigma^2$ given by Proposition~\ref{CovProp} is positive definite. Then, for $\mathbb P$-a.e. $\omega \in\Omega$ and every $\delta >0$, there exists a coupling between  $\{g_{\sigma^n\omega}\circ T_\omega^{n}: n\geq 0\}$, considered as a sequence of random variables on $(M,h_\omega^0)$,
and a  sequence $(Z_k)_k$ of independent centered (i.e. of zero mean) Gaussian random vectors
such that
\[
\bigg{\lvert} \sum_{i=0}^{n-1}g_{\sigma^n\omega}\circ T_\omega^{i}-\sum_{i=1}^n Z_i \bigg{\rvert} =O(n^{1/4+\delta}),\quad\text{almost-surely}.
\]
Moreover, there exists a constant $C=C_\delta(\omega)>0$ so that for every $n\geq1$,
\begin{equation*}
\left\|\sum_{i=0}^{n-1}g_{\sigma^n\omega}\circ T_\omega^{i}-\sum_{i=1}^n Z_i\right\|_{L^2}\leq Cn^{1/4+\delta}.
\end{equation*}
Furthermore, there is a constant $C'=C'_\delta(\omega)>0$ so that for every unit vector $v\in \mathbb R^d$,
$$
\left\|\sum_{i=0}^{n-1} g_{\sigma^n\omega}\circ T_\omega^{i}\cdot v\right\|_{L^2}^2-Cn^{1/2+\delta}\leq \left\|\sum_{i=1}^n Z_i\cdot v\right\|_{L^2}^2\leq \left\|\sum_{i=0}^{n-1} g_{\sigma^n\omega}\circ T_\omega^{i}\cdot v \right\|_{L^2}^2+C'n^{1/2+\delta}.
$$
\end{theorem}

\begin{proof}
The conclusion of the theorem follows from Proposition \ref{CovProp}, Lemma~\ref{631l} and Theorem~\ref{Gthm} (by noting that the sequence  $(g_{\sigma^n \omega}\circ T_\omega^{n} )_{n\in \N}$ is bounded in $L^p(M, h_\omega^0)$ for each $p\geq1$).
\end{proof}

\subsection{A general scheme for verifying condition~\eqref{Dec}}\label{subdec}
We will now discuss a  general framework under which~\eqref{Dec} holds true. In order to do so, we introduce some additional assumptions:
\begin{itemize}
\item $\mathcal B=(\mathcal B, \| \cdot \|)$ can be compactly embedded in $\mathcal B_w=(\mathcal B_w, \| \cdot \|_w)$ and $\| \cdot \|_w \le \| \cdot \|$ on $\mathcal B$;
\item there is a bounded linear operator $\mathcal L$ on $\mathcal B$ (that admits an extension to a bounded operator on $\mathcal B_w$) with the property that there exist $K, \rho >0$ such that
\begin{equation}\label{gap}
\|\mathcal L^n h\| \le Ke^{-\rho n}\|h \|, \quad \text{for $h \in \mathcal B$, $h(1)=0$ and $n \in \N$.}
\end{equation}
\item there exist $C_i>0$, $i \in \{1, 2, 3\}$ and $b\in (0, 1)$ such that 
\begin{equation}\label{ab1}
\| \mathcal L_n \cdots \mathcal L_1 \|_w \le C_1,
\end{equation}
and 
\begin{equation}\label{ab2}
\| \mathcal L_n \cdots \mathcal L_1 h\| \le C_2b^n \|h \|+ C_3 \|h \|_w, 
\end{equation}
for each $n\in \N$, $h \in \mathcal B$ and $\mathcal L_1,\ldots, \mathcal L_n \in \mathcal P$, where $\mathcal P$ is a family of bounded operators on $\mathcal B$, $\mathcal L \in \mathcal P$ such that  each element of $\mathcal P$ admits an extension to a bounded operator on $\mathcal B_w$. Moreover, $\{h \in \mathcal B: h(1)=0\}$ is invariant for each operator in $\mathcal P$.
\end{itemize}

For $\epsilon_0>0$, set
\[
\mathcal P(\mathcal L, \epsilon_0):=\bigg{\{} \mathcal L'\in \mathcal P: |||\mathcal L'-\mathcal L|||:=\sup_{\|h \| \le 1} \|(\mathcal L'-\mathcal L)h \|_w \le \epsilon_0 \bigg{\}}.
\]
Provided that $\epsilon_0>0$ is sufficiently small, it follows from~\cite[Proposition 2.10]{CR} that  there exist $D, \lambda >0$ such that
\[
\| \mathcal L_n \cdots \mathcal L_1 h \| \le De^{-\lambda n}\|h \|, \quad \text{for $h\in \mathcal B$, $h(1)=0$ and $n\in \N$.}
\]
Hence, if we build our cocycle $(\mathcal L_\omega)_{\omega \in \Omega}$ so that $\mathcal L_\omega \in \mathcal P(\mathcal L, \epsilon_0)$ for each $\omega \in \Omega$, we have that~\eqref{Dec} holds.

\section{Examples}\label{EX}
We now discuss various classes of random dynamical systems  to which Theorem~\ref{Th} is applicable. 
\subsection{Random hyperbolic dynamics}
Let $M$ be a $C^\infty$ compact connected Riemannian manifold and let $T$ be a topologically transitive Anosov map of class $C^{r+1}$, where $r>2$. For $\epsilon >0$,  set \[
\mathcal M_\epsilon(T):=\{S\colon M \to M : \ \text{$S$ is an Anosov map of class $C^{r+1}$ and $d_{ C^{r+1}}(S, T)<\epsilon$}\}.
\]
Let $\omega \to T_\omega \in \mathcal M_\epsilon(T)$ be a measurable map. By $\mathcal L_\omega$, we denote the transfer operator associated to $T_\omega$.  Provided that $\epsilon >0$ is sufficiently small, it is proved in~\cite[Section 3]{DFGTV} (by using arguments in Subsection~\ref{subdec}) that there exists a Banach space $\mathcal B=(\mathcal B, \| \cdot \|)$ and two norms $\| \cdot \|$ and $\|\cdot \|_w$ on $\mathcal B$, $\| \cdot \|_w 
\le \| \cdot \|$ such that~\eqref{LYineq} and~\eqref{Dec} hold. Moreover, elements of $\mathcal B$ are distributions of order at most $1$. The strong measurability of the map $\omega \mapsto \mathcal L_\omega$ is verified in~\cite[Subsection 3.1]{DFGTV}, while~\eqref{pm}
is proved in~\cite[Proposition 3.3]{DFGTV}.

Set $\mathcal O=C^r(M, \mathbb C)$. It follows from~\cite[Lemma 3.2]{GL} that~\eqref{prod} holds. Finally, \eqref{weakLY} can be established as in~\cite{DFGTV}.  Indeed, \eqref{weakLY} was proved in~\cite[p.653-654]{DFGTV} in the case when $d=1$, i.e.  when $g$ is  a real-valued observable. As for the case when $d>1$, it is sufficient to note that one only needs to justify that the version of~\cite[(58)]{DFGTV} holds true (with a constant independent on $\omega$, $n$ and $t$ with $\lvert t\rvert \le 1$). However, for this we only need to apply~\cite[(58)]{DFGTV} for $g_t:=t\cdot g$ instead of $g$, which can be done since 
\[
\sup_{\lvert t\rvert \le \rho}\|g_t(\omega,\cdot)\|_{C^r}\leq C_d\|g(\omega,\cdot)\|_{C^r},
\]
where $C_d>0$ is some constant which depends only on $d$. We conclude that Theorem~\ref{Th} can be applied in this setting. 

\begin{remark}\label{Demers}
We would like to explain the reason why we took $\mathcal O=C^r(M, \mathbb C)$.  Namely, this was done in order to satisfy assumptions introduced in Subsection~\ref{OS}, which in particular require that for $g\in \mathcal O$ and $h\in \mathcal B$, we can define 
$g\cdot h\in \mathcal B$. Moreover, for a fixed $g\in \mathcal O$, $h\mapsto g\cdot h$ needs to be a bounded operator on $\mathcal B$. 
These  properties  were crucial for the
twisted transfer operators $\mathcal L_\omega^{\theta}$
to be well defined and for our approach to
work. Since our $\mathcal B$ belongs to the class of anisotropic Banach spaces introduced in~\cite{GL} (which are precisely defined as closures of $C^r(M, \mathbb C)$ with respect to certain norms), it was necessary to consider observables in $C^r(M, \mathbb C)$.
We also note that the version of the central limit theorem case in the deterministic case  in \cite{GL} is also stated for $C^r$-observables (see \cite[Remark 2.10]{GL}).

As pointed to us by M. Demers, the class of anisotropic Banach spaces introduced
in  \cite{BT} have the property that the multiplication by a H\"older continuous observable
(with a suitably chosen H\"older constant) acts as a bounded linear operator. We note
that the same holds true for the class of spaces introduced Demers and Zhang \cite{DZ},
which are build to handle also the presence of singularities in the two-dimensional
setting. Consequently, using spaces introduced
in \cite{BT}, it seems reasonable that one can extend Theorem \ref{Th} by dealing with
observables $g$ which have the property that $g_\omega$ is H\"older continuous. We differ from
pursuing this direction because it would require a lot of additional technical work.
 We feel that this would somehow hide the main contribution of this paper,
which is the adaptation of the spectral approach for ASIP given in \cite{GO} to the case of quenched
random dynamics, and its applications.
\end{remark}

\subsection{Random perturbations of the Lorentz gas}
Let us consider the two dimensional torus $\mathbb T^2$ on which we  place finitely
many (disjoint) scatterers $\Gamma_i$, $i=1, \ldots, d$  which have $C^3$
 boundaries with strictly positive
curvature.
 We stress that in what follows these scatterers will be allowed to move but their number and the arclengths of their boundaries will not change. Set
\[
M:=\cup_{i=1}^d I_i \times [-\pi/2, \pi/2], 
\]
where $I_i$ is an interval with endpoints identified such that $|I_i|=|\partial \Gamma_i|$. Furthermore, let $m$ be the normalized Lebesgue measure on $M$, i.e. $dm=\frac{1}{\pi L}dr d\varphi$ where $L=\sum_{i=1}^d|I_i|$.
We consider the class $\mathcal F$ of maps on $M$  introduced in~\cite[Section 3.]{DZ}. 
We stress that  $\mathcal F$ contains various perturbations of the billiard map associated to periodic Lorentz gas (see~\cite[Section 2.4.]{DZ} for details). Let now $\mathcal F'$ consist of all those $T\in \mathcal F$ that preserve measure $\mu$ given by
$d\mu=\frac{\pi}{2} \cos \varphi\, dm$. Hence, for $T\in \mathcal F'$ we have that~\cite[(H5)]{DZ} holds with $\eta=1$. 

Let $\lVert \cdot \rVert_w$ and $\lVert \cdot \rVert_{\mathcal B}$ be  norms on $C^1(M, \mathbb C)$ introduced in~\cite[Section 3.2.]{DZ}. Moreover, let 
 $\mathcal B=(\mathcal B, \| \cdot \|)$ be the completion of $C^1 (M, \mathbb C)$  with respect to $\lVert \cdot \rVert_{\mathcal B}$. It follows from~\cite[Lemma 3.4.]{DZ} that elements of $\mathcal B$ are distributions of order at most $1$. 
Finally, let $d_{\mathcal F}$ be the distance on $\mathcal F$ introduced in~\cite[Section 3.4.]{DZ}.

Let us now fix $T\in \mathcal F'$ such that $(T, \mu)$ is mixing.  Thus, \eqref{gap} holds (with some $K, \rho>0$), where $\mathcal L=\mathcal L_T$.
For $\epsilon >0$, set
\[
\mathcal M_\epsilon (T):=\{ T' \in \mathcal F': d_{\mathcal F}(T, T')<\epsilon \}.
\]
Let $\omega \mapsto T_\omega \in \mathcal M_\epsilon(T)$ be a measurable map.  By~\cite[Theorem 2.3]{DZ}, there exist $C,\beta >0$ such that 
\[
||| \mathcal L_{T_1}-\mathcal L_{T_2} ||| =\sup_{\|h \| \le 1} \| (\mathcal L_{T_1}-\mathcal L_{T_2})h \|_w \le Ce^{\beta/2}, \quad \text{for $T_1, T_2 \in \mathcal F'$.}
\]
Moreover, it follows from the proof of~\cite[Proposition 5.6]{DZ} that~\eqref{ab1} and~\eqref{ab2} hold (with some $b\in (0, 1)$, $C_i>0$, $i\in \{1, 2, 3\}$), where $\mathcal L_i=\mathcal L_{T_i}$, $T_i \in \mathcal M_\epsilon (T)$.
Provided that $\epsilon >0$ is sufficiently small, it follows from the discussion in Subsection~\ref{subdec} that \eqref{Dec} holds. Moreover, the strong
measurability of $\omega \mapsto \mathcal L_\omega$ can be obtained by arguing as in~\cite[Subsection 3.1]{DFGTV}. In this setting $h_\omega^0=\mu$ and therefore~\eqref{pm} is trivially satisfied. 

Let $\gamma \in (0,1)$ be as in the statement of~\cite[Lemma 5.3.]{DZ} and let $\mathcal O=C^\gamma (M)$ be the space of all H\"{o}lder continuous functions $\varphi \colon M \to \mathbb C$ with H\"{o}lder  exponent $\gamma$.
Hence, \cite[Lemma 5.3.]{DZ} implies that~\eqref{prod} holds. Finally, one verifies~\eqref{weakLY} by arguing as in the previous subsection. We conclude that Theorem~\ref{Th} can be applied in this setting.

\subsection{Random Lasota-Yorke maps}\label{LY}
Our results hold true for the random expanding maps considered by Buzzi \cite{Buz}, for which the spectral method was developed in \cite{DFGTV2}. The results also hold true for the expanding maps considered in \cite[Ch. 5]{HK}. In order not to overload the paper we will only present a more concrete and relatively simple example. 

Let $M=[0,1]$ equipped with the Lebesgue measure $m$.  For a piecewise $C^2$ map $T\colon M\to M$, set $\delta (T)=\esinf_{x\in [0,1]} \lvert T'\rvert$ and let $b(T)$ denote the number of intervals of monoticity (branches) of $T$.
Consider now a measurable map $\omega\mapsto T_\omega$, $\omega \in \Omega$  of piecewise $C^2$ maps on $[0, 1]$ such that
\[
 b:=\esssup_{\omega \in \Omega} b(T_\omega)<\infty, \  \delta:=\esinf_{\omega \in \Omega} \delta (T_\omega)>1, \  D:=\esssup_{\omega \in \Omega}\lVert T''_\omega \rVert_{L^\infty}<\infty.
\]
In addition, we impose the following condition:
\[
\text{for every subinterval} \ J\subset M, \exists k= k(J) \text{ such that  for $\mathbb P$- a.e. }  \omega \in \Omega, \  T_\omega^{k}(J) = M.
\]
Finally, let $\mathcal B=(\mathcal B, \| \cdot \|)=(BV, \| \cdot \|_{BV})$ and $\| \cdot \|_w=\|\cdot \|_{L^1(m)}$. By identifying each $v\in BV$ with the functional $\phi \mapsto \int_M \phi v\, dm$ on $C^0(M, \mathbb C)$, we have that $\mathcal B$ consists of distributions of order $0$ (i.e. measures).
It is proved in~\cite[Subsection 2.3.1]{DFGTV2} that~\eqref{LYineq} and~\eqref{Dec} hold. Moreover, Proposition~\ref{ho} (see Remark~\ref{rpm}) implies that~\eqref{pm} holds. 

Set $\mathcal O=\mathcal B=BV$. Then, \eqref{prod} holds. Finally, \eqref{weakLY} holds with $B_3=1$. We conclude that Theorem~\ref{Th} can be applied in this setting.

\section{Acknowledgements}
We would like to express our gratitude to  anonymous referees and the handling editor for many useful comments and suggestions that helped us to improve our paper. In addition, we thank  Mark Demers for his comments in relation to Remark~\ref{Demers}.
D.D. was supported in part by  Croatian Science Foundation under the project IP-2019-04-1239 and by the University of Rijeka under the project uniri-prirod-18-9.

\bibliographystyle{amsplain}

\begin{thebibliography}{11}


\bibitem{ANV}
R. Aimino, M. Nicol and S. Vaienti, {\em Annealed and quenched limit theorems for
random expanding dynamical systems}, Probab. Theory Relat. Fields \textbf{162} (2015), 233-274.

\bibitem{Arnold}  L. Arnold, \emph{Random dynamical systems}, Springer Monogr. Math., Springer, Berlin,
1998.





\bibitem{BT}
V. Baladi and M. Tsujii, \emph{Anisotropic H\"older and Sobolev spaces for hyperbolic diffeomorphisms}, Ann. Inst. Fourier  \textbf{57} (2007), 127-154.

\bibitem{Buz}
J. Buzzi, {\em Exponential decay of correlations for random Lasota-Yorke maps}, Commun.
Math. Phys. \textbf{208} (1999),  25--54.

\bibitem{BP}
I. Berkes and W. Philipp, {\em  Approximation theorems for independent and weakly
dependent random vectors}, Ann. Probab. 7 (1979), 29--54.

\bibitem{CR}J.P. Conze and A. Raugi, \emph{Limit theorems for sequential expanding dynamical systems on $[0,1]$}, Ergodic theory and related fields, Contemp. Math., vol. 430, Amer. Math. Soc., Providence, RI, 2007, pp. 89–121.
\bibitem{CDKM} C. Cuny, J. Dedecker, A. Korepanov and F. Merlevede, \emph{Rates in almost sure invariance principle for slowly mixing dynamical systems}, Ergodic Theory Dynam. Syst. \textbf{40} (2020), 2317--2348.
\bibitem{CDKM2} C. Cuny, J. Dedecker, A. Korepanov and F. Merlevede, \emph{Rates in almost sure invariance principle for quickly mixing dynamical systems}, Stoch. Dyn. \textbf{20} (2020), 2050002, 28 pp.
\bibitem{CM} C. Cuny and F. Merlevede, \emph{Strong invariance principles with rate for reverse martingales and
applications}, J. Theoret. Probab. \textbf{28} (2015), 137--183. 
\bibitem{DZ} M. Demers and H. Zhang, \emph{A functional analytic approach to perturbations of the Lorentz Gas}, Comm. Math. Phys. \textbf{324} (2013), 767--830.




\bibitem{DFGTV1} D. Dragi\v cevi\' c, G. Froyland, C. Gonzalez-Tokman and S. Vaienti, \emph{Almost Sure Invariance Principle for random piecewise expanding maps}, Nonlinearity \textbf{31} (2018), 2252--2280.

\bibitem{DFGTV2} D. Dragi\v cevi\' c, G. Froyland, C. Gonzalez-Tokman and S. Vaienti, \emph{A spectral approach for quenched limit theorems for random expanding dynamical systems}, Comm. Math. Phys. \textbf{360} (2018), 1121--1187.

\bibitem{DFGTV} D. Dragi\v cevi\' c, G. Froyland, C. Gonzalez-Tokman and S. Vaienti, \emph{A spectral approach for quenched limit theorems for random hyperbolic dynamical systems}, Trans. Amer. Math. Soc. \textbf{373} (2020), 629--664.
\bibitem{DH}
D. Dragi\v cevi\' c and  Y. Hafouta \emph{Limit theorems for random expanding or Anosov dynamical systems and vector-valued observables}, Ann. Henri Poincare \textbf{21} (2020), 3869–3917.

\bibitem{DH1}
D. Dragi\v cevi\' c and  Y. Hafouta, \emph{Almost sure invariance principle for random distance expanding maps with a nonuniform decay of correlations}, Thermodynamic Formalism, CIRM Jean-Morlet Chair Subseries, Springer-Verlag, to appear. 
\bibitem{DS} D. Dragi\v cevi\' c and J. Sedro, \emph{Statistical stability and linear response for random hyperbolic dynamics}, preprint, https://arxiv.org/abs/2007.06088


\bibitem{FieldMelbourneTorok}
M. Field, I. Melbourne and A. T\"or\"ok,
\emph{Decay of correlations, central limit theorems and approximation by
   Brownian motion for compact Lie group extensions},
   Ergodic Theory Dynam. Systems,
   {\bf 23}   (2003), 87--110.

\bibitem{GL} S. Gou\"{e}zel and C. Liverani, \emph{Banach spaces adapted to Anosov systems}, Ergodic Theory Dynam. Systems \textbf{26} (2006), 123--151. 
\bibitem{GO} S. Gou\"ezel, \emph{Almost sure invariance principle for dynamical systems by spectral methods},  Annals of Probability  \textbf{38}  (2010), 1639--1671.

\bibitem{HK} 
Y. Hafouta and Y. Kifer, {\em Nonconventional limit theorems and random dynamics},
World Scientific, 2018.


\bibitem{HNTV} N.  Haydn, M. Nicol, A. T\"{o}rok and  S. Vaienti, \emph{Almost sure invariance principle for sequential and non-stationary dynamical systems}, Trans. Amer. Math. Soc. \textbf{369} (2017), 5293--5316.
\bibitem{IL}
I.A. Ibragimov and Y.V. Linnik, {\em Independent and Stationary Sequences of Random Variables}, Wolters-Noordhoff, Groningen, 1971



\bibitem{HR} D. L. Hanson and R. P. Russo, \emph{Some Results on Increments of the Wiener Process with Applications to Lag Sums of I.I.D. Random Variables}, Ann. Probab. \textbf{11} (1983), 609--623.


\bibitem{kifer} Y. Kifer, \emph{Limit theorems for random transformations and processes in random environments}, Trans. Amer. Math. Soc. \textbf{350} (1998), 1481--1518.



 \bibitem{KorepanovEq}
A. Korepanov,
\emph{Equidistribution for Nonuniformly Expanding Dynamical Sys- tems, and Application to the Almost Sure Invariance Principle}, Comm. Math. Phys. \textbf{359} (2018),  1123-1138.

\bibitem{Korepanov}
A. Korepanov,
\emph{Rates in almost sure invariance principle for Young towers with exponential tails},
  Comm. Math. Phys. \textbf{363} (2018),  173--190.
  

  

\bibitem{KZM}
A Korepanov, Z Kosloff, I Melbourne
{\em Martingale-coboundary decomposition for families of dynamical systems},
Annales Inst. H. Poincar\'e Analyse Non Lin\'eaire \textbf{35} (2018),
859-885.
 \bibitem{MN1} I. Melbourne and M. Nicol, \emph{Almost sure invariance principle for nonuniformly hyperbolic systems}, Commun. Math. Phys. {\bf 260}, 131--146, (2005).
\bibitem{MN2} I. Melbourne and M. Nicol, \emph{A vector-valued almost sure invariance principle for hyperbolic dynamical systems}, Annals of Probability {\bf 37}, 478-505, (2009).



 

 
 
\bibitem{PS}
W. Philipp and W.F. Stout, {\em Almost sure invariance principles for partial sums of weakly dependent random variables}, Mem. Amer. Math. Soc. \textbf{161} (1975).
    
\bibitem{STE}    
M. Stenlund,  \emph{A vector-valued almost sure invariance principle for Sinai billiards with random
scatterers}, Commun. Math. Phys. \textbf{325}  (2014), 879-916.

\bibitem{STSU}
M. Stenlund and  H. Sulku, {\em A coupling approach to random circle maps expanding
on the average}, Stoch. Dyn. \textbf{14} (2014), 1450008, 29 pp.

\bibitem{Su1}
Y. Su, {\em Random Young towers and quenched limit laws}
preprint, arXiv 1907.12199.


\bibitem{Su2}
Y. Su, {\em Vector-valued almost sure invariance principle  for non-stationary dynamical systems}
preprint, arXiv 1903.09763.



\bibitem{Zai}
A.Y. Zaitsev, \emph{Estimates for the rate of strong approximation in the multidimensional
invariance principle}. J. Math. Sci. (2007) 4856--4865.
    
\end{thebibliography}

\end{document}